\documentclass[letter,journal,final]{IEEEtran}  
\IEEEoverridecommandlockouts       
\usepackage{amsthm} 

\usepackage{cite}
\usepackage{amsmath,amssymb,amsfonts,mathrsfs}
\usepackage{algorithm,algorithmic,setspace}
\usepackage{graphicx}
\usepackage{textcomp}
\usepackage{epstopdf} 
\usepackage{graphics} 
\usepackage{subfig}
\usepackage{epstopdf} 
\usepackage{url}
\usepackage{color,soul}
\usepackage[hidelinks, colorlinks=true,linkcolor=blue]{hyperref}

\newtheorem{thm}{Theorem}
\newtheorem{lem}{Lemma}
\newtheorem{prop}{Proposition}

\newtheorem{assum}{Assumption}
\newtheorem{rem}{Remark}

\DeclareMathOperator{\ad}{ad}
\DeclareMathOperator{\Ad}{Ad}
\DeclareMathOperator{\tr}{tr} 
\DeclareMathOperator{\dom}{dom}  
\DeclareMathOperator{\diag}{diag}
\newcommand{\T}{^{\top}} 

\newcommand{\ie}{\textit{i.e.}}

\allowdisplaybreaks

\title{
	Hybrid Feedback for  Global Tracking on Matrix Lie Groups $SO(3)$ and $SE(3)$ 
} 
\author{Miaomiao Wang  and Abdelhamid Tayebi  
	\thanks{This work was supported by the National Sciences and Engineering Research Council of Canada (NSERC), under the grants NSERC-DG RGPIN-2020-06270. Preliminary results of this paper were presented at the 58th IEEE Conferene on Decision and Control, Nice, France, December 2019 \cite{wang2019new}. (\textit{Corresponding author: Miaomiao Wang.})
} 
	\thanks{Miaomiao Wang is with the Department of Electrical and Computer Engineering,
		Western University, London, ON N6A 3K7, Canada (e-mail: {\tt\small mwang448@uwo.ca}).}%
	\thanks{Abdelhamid Tayebi is with the Department of Electrical Engineering, Lakehead University, Thunder Bay, ON P7B 5E1, Canada. He is also with the Department of Electrical and Computer Engineering, 
		Western University, London, ON N6A 3K7, Canada.
		(e-mail: {\tt\small atayebi@lakeheadu.ca}).}
}%

\begin{document}

	\maketitle
	
	
\begin{abstract}
		We introduce a new hybrid control strategy, which is conceptually different from the commonly used  synergistic hybrid approaches, to efficiently deal with the problem of the undesired equilibria that precludes smooth vectors fields on $SO(3)$ from achieving global stability. The key idea consists in constructing a suitable potential function on $SO(3)\times \mathbb{R}$ involving an auxiliary scalar variable, with flow and jump dynamics, which keeps the state away from the undesired critical points while, at the same time, guarantees a decrease of the potential function over the flows and jumps. Based on this new hybrid mechanism, a hybrid feedback control scheme for the attitude tracking problem on $SO(3)$, endowed with global asymptotic stability and semi-global exponential stability guarantees, is proposed. This control scheme is further improved through a smoothing mechanism that removes the discontinuities in the input torque. The third hybrid control scheme, proposed in this paper, removes the requirement of the angular velocity measurements, while preserving the strong stability guarantees of the first hybrid control scheme. This approach has also been applied to the tracking problem on $SE(3)$ to illustrate its advantages with respect to the existing synergistic hybrid approaches. Finally, some simulation results are presented to illustrate the performance of the proposed hybrid controllers.	
	\end{abstract}
	
	
	\section{Introduction}
	The attitude tracking control problem of rigid body systems has been widely investigated in the literature with many applications related to robotics, aerospace and marine engineering, for instance \cite{crouch1984spacecraft,Wie1989,Wen1991, pettersen1999time,chaturvedi2011rigid,naldi2017robust}.  In particular, geometric control design on Lie groups $SO(3)$ and $SE(3)$, has generated a great deal of research work   \cite{koditschek1989application,bullo1999tracking,bhat2000topological,maithripala2006almost,chaturvedi2011rigid,lee2012robust,invernizzi2019dynamic}. It is well known that achieving global stability results with feedback control schemes designed on Lie groups such as $SO(3)$ and $SE(3)$, is a fundamentally difficult task due to the topological obstruction of the motion space induced by the fact that these manifolds are not homeomorphic to $\mathbb{R}^n$ and that there is no smooth vector field that can have a global attractor \cite{koditschek1989application}. To achieve almost global asymptotic stability (AGAS), a class of \textit{suitable} ``navigation functions" has been introduced in \cite{koditschek1989application}. In \cite{bullo1999tracking}, the Riemannian structure of the configuration manifold for a class of mechanical systems, is used to derive a local exponential control law, while in \cite{maithripala2006almost}, an almost global tracking controller has been proposed for a general class of Lie groups via intrinsic globally-defined error dynamics. It is well known that for any smooth potential function on $SO(3)$,  there exist at least four critical points where its gradient vanishes \cite{morse1934calculus}, and as such, almost global stability is the strongest result that one can achieve in this case.

	Using the hybrid dynamical systems framework of \cite{goebel2009hybrid,goebel2012hybrid}, a unit-quaternion based hybrid control scheme for global attitude tracking was first proposed in \cite{mayhew2011quaternion}, which led thereafter to other variants such as  \cite{schlanbusch2012hybrid,schlanbusch2014hybrid,gui2016global}. The use of a ``synergistic" family of potential functions to overcome the topological obstruction to global asymptotic stability (GAS) on $SO(3)$ has been introduced in \cite{mayhew2011synergistic}. A family of potential functions is ``centrally'' synergistic, if the identity is the common critical point of all the potential functions in the family.  This synergistic hybrid approach was successfully applied to the rigid body attitude control problem in \cite{mayhew2011hybrid}, where a hysteresis-based switching mechanism was introduced to avoid all the undesired critical points and ensure some robustness to measurement noise. However, only  numerical examples were provided to construct such a synergistic family of potential functions via angular warping on $SO(3)$. Inspired by this, a number of hybrid controllers and observers on Lie groups $SO(3)$, $SE(3)$ and $SE_2(3)$ have been proposed in the literature \cite{mayhew2013synergistic,casau2015globally,lee2015global,berkane2017construction,berkane2018hybrid,berkane2017hybrid,berkane2017hybrid2,wang2019hybrid,wang2020hybrid}.  The work in \cite{berkane2017construction}, provides a systematic and comprehensive procedure for the construction of synergistic potential functions on $SO(3)$, which is then applied to design velocity-free hybrid attitude stabilization and tracking control schemes \cite{berkane2017construction,berkane2018hybrid}. Moreover, a hybrid attitude control scheme on $SO(3)$ using an ``exp-synergistic'' family of potential functions, leading to global exponential stability, has been proposed in \cite{berkane2017hybrid}.  Alternatively, a ``non-central" synergistic family of potential functions has been considered in \cite{mayhew2013synergistic} to relax the centrality condition. A similar idea can be found in \cite{lee2015global}. Recently, a  hybrid control approach on smooth manifolds, relying on a switching logic between local coordinates, has been proposed in \cite{casau2020hybrid}.

	\textit{Contributions:} In this paper, we propose new hybrid feedback control strategies for the  tracking problems on matrix Lie groups $SO(3)$ and $SE(3)$, leading to GAS guarantees. Some extensions are also proposed to smooth out the discontinuities induced by the control input  switching, and to handle the lack of angular velocity measurements. The main novelty and strength of our approach is the fact that it can overcome the compactness condition required in the synergistic hybrid approaches. Therefore, it can be applied to a general class of non-compact matrix Lie groups to generate globally asymptotically stabilizing hybrid feedback laws as demonstrated through the design of a geometric hybrid tracking control scheme on the non-compact manifold $SE(3)$.
	The key idea of our hybrid attitude control strategy consists in using  a suitable  potential function on $SO(3)\times \mathbb{R}$ involving an auxiliary scalar variable with hybrid dynamics. This scalar variable is governed by some appropriately designed flow dynamics when the state in $SO(3)\times \mathbb{R}$ is away from the undesired critical points, and is governed by an appropriately designed jump strategy when the state is in the neighborhood of the undesired critical points. The flow and jump strategies are designed to ensure a decrease of the potential function over the flows and jumps. In contrast with the synergistic hybrid approaches where a synergistic family of  potential functions on $SO(3)$ is used \cite{mayhew2011hybrid,mayhew2011synergistic,mayhew2013synergistic,berkane2017construction,berkane2018hybrid}, our hybrid approach relies on a single potential function on $SO(3)\times \mathbb{R}$ parameterized by the hybrid auxiliary scalar variable. As it is going to be shown later, our approach on top of the hybrid control design simplification, allows to overcome the set compactness assumption (stemming from the diffeomorphism condition of the transformation map) required in the synergistic hybrid approaches. This important advantage, makes our approach a good candidate for the design of hybrid controllers on non-compact manifold such as $SE(3)$ where the existing hybrid approaches (for instance, \cite{mayhew2011hybrid,mayhew2011synergistic,lee2015global,berkane2017construction,berkane2017hybrid,berkane2018hybrid,casau2020hybrid}) are not applicable in a straightforward manner.  
	A preliminary version of this work has been presented in \cite{wang2019new} without the semi-global exponential stability proof, the control smoothing mechanism and the velocity-free hybrid tracking scheme presented in the present paper.

	\textit{Organization:} The remainder of this paper is organized as follows: Section \ref{sec:backgroud} provides the preliminaries that will be used throughout this paper. In Section \ref{sec:problem}, we formulate our attitude tracking control problem. In Section \ref{sec:min-resetting}, a new hybrid mechanism using a potential function on $SO(3) \times \mathbb{R}$ is presented. In Sections \ref{sec:basic}-\ref{sec:velocity_free}, we present our three hybrid attitude tracking controllers.
	In Section \ref{sec:potential}, we provide a systematic procedure for the construction of the  potential function on $SO(3)\times \mathbb{R}$ satisfying all the requirements for the design of our hybrid attitude tracking controllers. In Section \ref{sec:SE(3)}, our hybrid control strategy is extended to the pose tracking problem on the non-compact Lie group $SE(3)$. Simulation results are presented in Section \ref{sec:simulation}.

	\section{Preliminaries} \label{sec:backgroud}
	\subsection{Notations and Definitions}
	The sets of real, non-negative real and natural numbers are denoted by $\mathbb{R}$, $\mathbb{R}_{\geq 0}$ and $\mathbb{N}$, respectively. We denote by $\mathbb{R}^n$ the $n$-dimensional Euclidean space, and by $\mathbb{S}^n$ the set of unit vectors in $\mathbb{R}^{n+1}$. Given two matrices, $A,B\in \mathbb{R}^{m\times n}$, their Euclidean inner product is defined as $\langle\langle A,B\rangle\rangle = \tr(A\T B)$. The Euclidean norm of a vector $x\in \mathbb{R}^n$ is denoted by $\|x\|$, and the Frobenius norm a matrix $X\in \mathbb{R}^{n\times m}$ is denoted by $\|X\|_F = \sqrt{\tr(X\T X)}$. 
	For a given symmetric matrix $A\in \mathbb{R}^{n\times n}$, we define $\mathcal{E}(A)$ as the set of all the unit-eigenvectors of $A$, $(\lambda^A_i, v_i^A)$ as the $i$-th pair of eigenvalue and eigenvector of $A$, and $\lambda_{\min}^A$ and $\lambda_{\max}^A$ as the minimum and maximum eigenvalue of $A$, respectively. 
	Let  $\mathcal{M}$ be a  smooth manifold embedded in $\mathbb{R}^n$ and $T_x\mathcal{M}$ be the \textit{tangent space} on $\mathcal{M}$ at point $x\in \mathcal{M}$. The \textit{gradient} of a  differentiable real-valued function $f:\mathcal{M} \to  \mathbb{R}$ at point $x\in \mathcal{M}$, denoted by $\nabla_x f(x) \in T_x\mathcal{M}$, relative to the \textit{Riemannian metric} $\langle~,~ \rangle_x: T_x\mathcal{M} \times T_x\mathcal{M} \to \mathbb{R}$ is uniquely defined by $\dot{f}(x) = \langle \nabla_x f(x), \xi\rangle_x$ for all $\xi\in T_x \mathcal{M}$.  A point $x\in \mathcal{M}$ is called a \textit{critical point}  of $f$ if its  gradient  at $x$ varnishes  (\ie, $\nabla_{x} f(x) = 0$).  A continuously differentiable function $f:\mathcal{M} \to \mathbb{R}_{\geq 0} $ is said to be a \textit{potential function} on $\mathcal{M}$ with respect to the  set $\mathcal{A} \subset \mathcal{M}$ if $f(x)>0$ for all $x\notin  \mathcal{A}$ and $f(x) = 0$ for all $x\in \mathcal{A}$.

	The 3-dimensional  \textit{Special Orthogonal group} is defined by  
	$
	SO(3): = \left\{R\in \mathbb{R}^{3\times 3}: R\T R = RR\T = I_3, \det(R) = +1\right\}
	$, and  the \textit{Lie algebra} of  $SO(3)$ is defined by
	$
	\mathfrak{so}(3) := \left\{\Omega\in \mathbb{R}^{3\times 3}: \Omega\T=-\Omega \right\}
	$. The tangent space of $SO(3)$ at any base point $R$ is  defined by
	$
	T_R SO(3) : = \{R\Omega\in \mathbb{R}^{3\times 3} :\ R\in SO(3), \Omega \in \mathfrak{so}(3)\}
	$. The inner product  in the tangent space $T_R SO(3)$ defines the left invariant metric on $SO(3)$ as $\langle R \Omega_1,R \Omega_2 \rangle_R = \langle \langle  \Omega_1,  \Omega_2 \rangle \rangle$ for all $R\in SO(3)$ and $\Omega_1,\Omega_2\in \mathfrak{so}(3)$.
	For any $R\in SO(3)$, we define $|R|_I\in [0,1]$ as the normalized Euclidean distance on $SO(3)$ with respect to the identity $I_3$, which is given by $|R|_I^2 = \tr(I_3-R)/4$. 
	Let the map $\mathcal{R}_a: \mathbb{R} \times \mathbb{S}^2\to SO(3)$ represent the well-known angle-axis parameterization of the attitude, which is given by
	$
	\mathcal{R}_a(\theta,u) :=  I_3 + \sin(\theta) u^\times + (1-\cos(\theta))(u^\times)^2
	$
	with $\theta\in \mathbb{R}$ denoting the
	rotation angle and $u\in \mathbb{S}^2$ denoting the rotation axis. For a given vector $x:=[x_1,x_2,x_3]\T \in \mathbb{R}^3$, we define $x^\times$ as the skew-symmetric matrix given by
	\[
	x^\times = \begin{bmatrix}
		0 & -x_3 & x_2 \\
		x_3 & 0 & -x_1 \\
		-x_2 & x_1 & 0
	\end{bmatrix} 
	\] 
	and $\text{vec}(\cdot)$ as the inverse operator of the map $(\cdot)^\times$, such that $\text{vec}(x^\times) = x$.	
	For a matrix $A\in \mathbb{R}^{3\times 3}$, we denote   $\mathbb{P}_a(A) := \frac{1}{2}(A-A\T)$ as the anti-symmetric projection of $A$. Define the composition map $\psi: =\text{vec} \circ \mathbb{P}_a $  such that, for a matrix $A=[a_{ij}] \in \mathbb{R}^{3\times 3}$, one has
	$
	\psi (A) := \text{vec} (\mathbb{P}_a(A)) =\frac{1}{2}[
	a_{32} - a_{23},
	a_{13} - a_{31},
	a_{21} - a_{12}
	]\T.
	$
	For any $A\in \mathbb{R}^{3\times 3}, x\in \mathbb{R}^3$, one can verify that
	$
	\langle\langle A, x^\times\rangle\rangle
	= 2x\T \psi (A).
	$

	\subsection{Hybrid Systems Framework} 
	Consider a  smooth manifold $\mathcal{M}$ embedded in $\mathbb{R}^n$ and its tangent bundle denoted by $T \mathcal{M}=\bigcup_{x\in \mathcal{M}} T_x \mathcal{M}$.
	A general model of a hybrid system is given as \cite{goebel2009hybrid}:
	\begin{equation}\mathcal{H}:  
		\begin{cases}
			\dot{x} ~~\in F(x),& \quad x \in \mathcal{F}   \\
			x^{+} \in G(x),& \quad x \in \mathcal{J}
		\end{cases} \label{eqn:hybrid_system}
	\end{equation}
	where $x\in \mathcal{M}$ denotes the state, $x^+$ denotes the state after an instantaneous jump, the \textit{flow map} $F: \mathcal{M} \to T \mathcal{M}$ describes the continuous flow of $x$ on the \textit{flow set} $\mathcal{F} \subseteq \mathcal{M}$, and the \textit{jump map} $G: \mathcal{M}\rightrightarrows  \mathcal{M}$ (a set-valued mapping from $\mathcal{M}$  to $\mathcal{M}$) describes the discrete jumps of $x$ on the \textit{jump set} $\mathcal{J} \subseteq \mathcal{M}$.  A solution $x$ to $\mathcal{H}$ is parameterized by $(t, j) \in \mathbb{R}_{\geq 0} \times \mathbb{N}$, where $t$ denotes the amount of time passed and $j$ denotes the number of discrete jumps that have occurred.  A subset $\dom x \subset \mathbb{R}_{ \geq 0} \times \mathbb{N}$ is a \textit{hybrid time domain} if for every $(T,J)\in \dom x$, the set, denoted by $\dom x \bigcap ([0,T]\times \{0,1,\dots,J\})$, is a union of finite intervals  of the form $\bigcup_{j=0}^{J} ([t_j,t_{j+1}] \times \{j\})$ with  a time sequence $0=t_0 \leq t_1 \leq \cdots \leq t_{J+1}$. 
	A solution $x$ to $\mathcal{H}$ is said to be  \textit{maximal} if
	it cannot be extended by flowing nor jumping, and \textit{complete} if its domain $\dom x$ is unbounded. Let $|x|_{\mathcal{A}}$  denote the distance of a point $x$ to a closed set $\mathcal{A} \subset \mathcal{M}$,  and then the set $\mathcal{A}$ is said to be: 
	\textit{stable} for $\mathcal{H}$ if for each $\epsilon>0$ there exists $\delta>0$ such that each maximal solution $x$ to $\mathcal{H}$ with $|x(0,0)|_{\mathcal{A}} \leq \delta$ satisfies $|x(t,j)|_{\mathcal{A}} \leq \epsilon$ for all $(t,j)\in \dom x$; \textit{globally attractive} for $\mathcal{H}$ if  every maximal solution $x$ to $\mathcal{H}$  is complete and satisfies $\lim_{t+j\to \infty}|x(t,j)|_{\mathcal{A}} = 0$ for all $(t,j)\in \dom x$; \textit{globally asymptotically stable} if it is both stable and globally attractive for $\mathcal{H}$. 
	Moreover, the $\mathcal{A}$ is said to be (locally) \textit{exponentially stable} for   $\mathcal{H}$ if there exist $\kappa, \lambda, \mu>0$ such that, for any $|x(0,0)|_{\mathcal{A}} < \mu$, every maximal solution  $x$  to $\mathcal{H}$ is complete  and satisfies 
	$|x(t,j)|_{\mathcal{A}} \leq \kappa \exp(-\lambda(t+j))|x(0,0)|_{\mathcal{A}}$ for all $(t,j)\in \dom x$ \cite{teel2013lyapunov}. We refer the reader to \cite{goebel2009hybrid,goebel2012hybrid} and references therein for more details on hybrid dynamical systems.

	\section{Problem Formulation}  \label{sec:problem}

	The dynamical equations of motion of a rigid body on $SO(3)$ are given by
	\begin{align}
		\begin{cases}
			\dot{R} &= R\omega^\times   \\
			J\dot{\omega} &= - \omega^\times J\omega + \tau  
		\end{cases} \label{eqn:dynamics_R}
	\end{align}
	where the rotation matrix $R$ denotes the attitude of the rigid body, $\omega\in \mathbb{R}^3$ is the body-frame angular velocity, $J=J\T\in \mathbb{R}^{3\times 3}$ is the inertia matrix (constant and known), and $\tau\in \mathbb{R}^3$ is the control torque to be designed. 

	Let $m>0$ and let $\mathcal{W}_d \subset SO(3)\times \mathbb{R}^3$ be a compact (closed and bounded) subset. Let the desired reference trajectory be generated by the following dynamical system \cite{mayhew2013synergistic}: 
	\begin{align}
		\left. \begin{array}{rl}
			\dot{R}_r  &= R_r \omega_r^\times\\
			\dot{\omega}_r & = z \\
			z  &\in m\mathbb{B}
		\end{array}\right\} (R_r,\omega_r)\in \mathcal{W}_d \label{eqn:dynamics_R_r}
	\end{align}
	where  $R_r$ and $\omega_r$ are the desired rotation and angular velocity, and $m\mathbb{B}:=\{x\in \mathbb{R}^3: \|x\|\leq m\}$ is the closed ball in $\mathbb{R}^3$.
	As shown in \cite{mayhew2013synergistic},  every maximal solution to \eqref{eqn:dynamics_R_r} is complete, and any possible solution component $\omega_r$ of \eqref{eqn:dynamics_R_r} is Lipschitz continuous with Lipschitz constant $m$, but not necessarily differentiable. 
	
	Let us consider the left-invariant attitude tracking error $R_e= R_r\T R$  and the angular velocity tracking error $\omega_e = \omega - R_e\T \omega_r$. From \eqref{eqn:dynamics_R}-\eqref{eqn:dynamics_R_r}, one obtains the following error dynamics \cite{mayhew2013synergistic}:
	\begin{subequations} \label{eqn:Re-Jomega_e}
		\begin{align}
			\dot{R}_e & = R_e \omega_e^\times  \label{eqn:Re}\\
			J\dot{\omega}_e  &=  \Sigma(R_e,\omega_e,\omega_r)\omega_e  - \Upsilon(R_e, \omega_r,z) + \tau \label{eqn:Jomega_e} 
		\end{align}
	\end{subequations}
	where the functions $\Upsilon: SO(3)\times \mathbb{R}^3 \times \mathbb{R}^3 \to \mathbb{R}^3$ and  $\Sigma: SO(3)\times \mathbb{R}^3 \times \mathbb{R}^3 \to \mathfrak{so}(3)$ are given by
	\begin{subequations}
		\begin{align}
			\Upsilon(R_e,\omega_r,z) &= JR_e\T z + (R_e\T \omega_r)^\times JR_e\T \omega_r  \label{eqn:Upsilon_def} \\
			\Sigma(R_e,\omega_e,\omega_r)  &= (J\omega_e)^\times  + (JR_e\T \omega_r)^\times  \nonumber \\
			&\qquad  - ((R_e\T \omega_r)^\times J + J(R_e\T \omega_r)^\times). \label{eqn:Sigma_def}
		\end{align}
	\end{subequations}
	Note that  $\Sigma$ is skew-symmetric, and as such, for each $u\in \mathbb{R}^3$ one has $u\T \Sigma(R_e,\omega_e,\omega_r)  u = 0$. It is clear that $\Upsilon(R_e,\omega_r,z)$ is known,  and  $\Upsilon(R_e,\omega_r,z)  = 0$ if   $R_r$ is constant.

	The control objective consists in designing hybrid feedback laws such that the desired equilibrium of the error dynamics \eqref{eqn:Re}-\eqref{eqn:Sigma_def}, \ie, $(R_e = I_3, \omega_e = 0)$, is globally asymptotically stable.

	\section{A new Hybrid Mechanism Using a Single Potential Function on $SO(3) \times \mathbb{R}$}	\label{sec:min-resetting} 
	Let  $V(R)$ be a potential function on $SO(3)$ with respect to $I_3$. Let $\nabla_R V : R \to T_RSO(3)$ denote the gradient of $V$ at point $R$. 
	According to Lusternik-Schnirelmann theorem \cite{lusternik1934methodes} and Morse theory \cite{morse1934calculus}, a smooth vector field on $SO(3)$ can not have a global attractor, and any smooth potential function on $SO(3)$ must have at least four critical points. Let the set of all critical points of $V(R)$ be denoted by $ \Psi_V = \{R\in SO(3)|   \nabla_R V(R) = 0 \}  $, and the set of all \textit{undesired} critical points be denoted by $ \Psi_V\setminus\{I_3\}$.

	One possible way to make the desired critical point $I_3$ a global attractor, consists in generating a non-smooth gradient-based vector field on $SO(3)$ through a switching mechanism between a family of potential functions as done in \cite{mayhew2011hybrid,mayhew2011synergistic,berkane2017construction}. The potential functions are constructed using a modified trace function and a transformation map $\mathcal{T}: SO(3)\to SO(3)$ such that all the potential functions share only the desired critical point $I_3$. 
	For instance, a transformation map $\mathcal{T}(R) = \exp(\vartheta (R) u^\times )R$ with $\vartheta (R)=k V(R), k\in \mathbb{R}, u\in \mathbb{S}^2$, known as the ``angular warping",  is considered in \cite{mayhew2011synergistic}. As shown in \cite[Theorem 8]{mayhew2011synergistic}, $\mathcal{T}$ is required to be a diffeomorphism, and to be as such, $k$ has to be chosen sufficiently small in magnitude (\ie, $\sqrt{2}|k| \max\|\nabla_{R}V(R)\|_F < 1$ for all $R\in SO(3)$). A similar design of the central synergistic family of potential functions can be found in \cite{berkane2017construction}, where a different transformation map $\mathcal{T}(R) = R\exp(\vartheta(R) u^\times )$ with $\vartheta(R)=2\arcsin(k V(R))$ and $k$ sufficiently small, has been considered. Note that the existence of the parameter $k$ in \cite{mayhew2011hybrid,mayhew2011synergistic,berkane2017construction} is guaranteed mainly due to the fact that $SO(3)$ is compact. Alternatively, in \cite{mayhew2013synergistic}, a non-central synergistic family of potential functions has been designed based on a modified trace function through constant translation, scaling and biasing. Unfortunately, it is not straightforward to explicitly construct such a family of potential functions, especially when dealing with systems evolving on non-compact manifolds.

	To overcome the above mentioned problems, we propose a new approach that does not require the generation of a family of potential functions via a diffeomorphism map, leading to a much simpler design of hybrid control systems on $SO(3)$ or other non-compact manifolds. The key idea consists in using a single potential function $U: SO(3)\times \mathbb{R} \to \mathbb{R}_{\geq 0}$,  with respect to the  $\mathcal{A}_o:= (I_3,0)$, parameterized by a scalar variable $\theta\in \mathbb{R}$ that has flow and jump dynamics. In contrast with the previously mentioned synergistic potential functions approaches, where the potential functions are parameterized by a discrete variable, our single potential function $U$ is adjusted through the continuous flows and the discrete jumps of the auxiliary variable $\theta$ such that the resulting non-smooth gradient-based vector field yields a single attractor $\mathcal{A}_o$. The details of the construction of the potential function  $U$ and its properties will be presented later. 
	Let us define the set of all the critical points of $U$   as  
	\begin{multline}
		\Psi_U : =  \{(R,\theta)\in SO(3)\times \mathbb{R}:  \\ 
		\nabla_{R} U(R,\theta)  =0_{3\times 3},  
		\nabla_{\theta} U(R,\theta) =0\}  \label{eqn:criticalU}
	\end{multline} 
	with  $\nabla_{R} U(R,\theta)$ and $\nabla_{\theta} U(R,\theta)$ denoting  the gradients of $U(R,\theta)$ with respect to $R$ and $\theta$, respectively. Let $\Theta \subset \mathbb{R}$ be a nonempty and finite real set, and consider the following basic assumption for our potential function $U(R,\theta)$:
	\begin{assum}[Basic Assumption]\label{assum:1}
		There exist a potential function $U$ on $SO(3)\times \mathbb{R}$ with respect to $\mathcal{A}_o$  and a nonempty finite set  $\Theta \subset \mathbb{R}$ such that  $\mathcal{A}_o \in \Psi_U$  and
		\begin{equation}
			\mu_U(R,\theta)  > \delta,  \quad  \forall (R,\theta)\in \Psi_U\setminus\{\mathcal{A}_o\}  \label{eqn:u_U_delta}
		\end{equation} 
		with some constant $\delta >0$ and the map $\mu_U:SO(3)\times \mathbb{R}  \to \mathbb{R}_{\geq 0}$ is given by
		\begin{equation}
			\mu_U(R,\theta) := U(R,\theta) - \min_{{\theta}'\in \Theta}U(R,{\theta}')  \label{eqn:defu_U}.
		\end{equation}
	\end{assum}
	\begin{rem}
		From the definitions of sets $\Psi_U$ and $\mathcal{A}_o$, the set $\Psi_U\setminus\{\mathcal{A}_o\}$ denotes the set of all the undesired critical points of $U(R,\theta)$.  
		Assumption 1 implies that for any undesired critical point $(R,\theta) \in \Psi_U\setminus\{\mathcal{A}_o\}$, there exists another state $(R,{\theta}')$ with ${\theta}'\in \Theta$  such that the value of $U(R,{\theta}')$ is lower than the value of $U(R,\theta)$ by a constant gap $\delta$. Hence, one can reset $\theta$ (at each jump) to the one leading to the minimum value of $U$ with a strict decrease such that the state after jump is away from the undesired critical points. This property, together with an appropriately designed feedback over the flows, will guarantee global asymptotic stability of desired equilibrium point (see an example in Fig. \ref{fig:diagram1}).
		This basic assumption is motivated from the synergistic family of potential functions  on $SO(3)$  proposed in \cite{mayhew2011hybrid,mayhew2011synergistic,mayhew2013synergistic}. 
	\end{rem}

	\begin{figure}[!ht]
		\centering
		\includegraphics[width=0.7\linewidth]{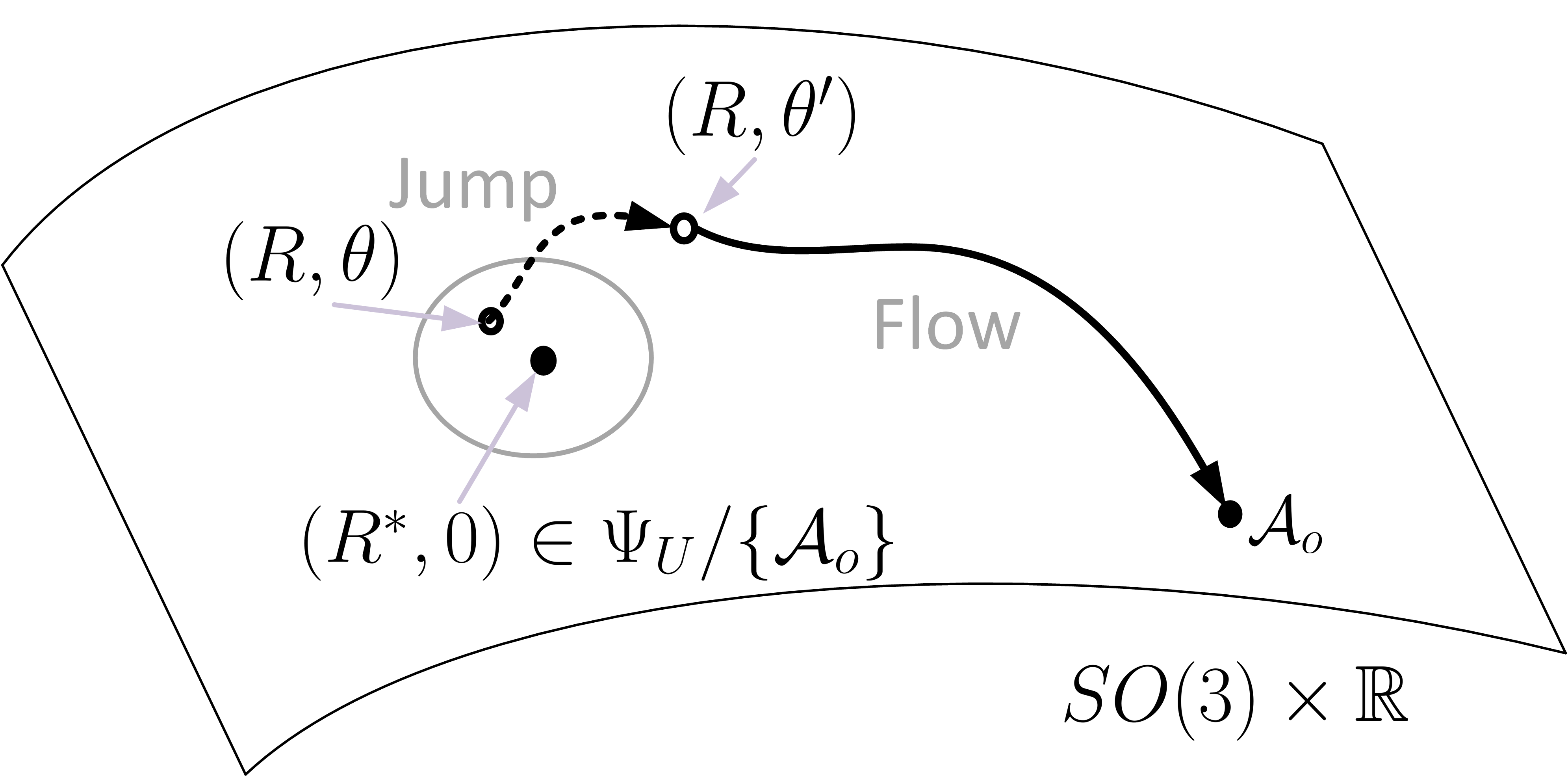} 
		\caption{Hybrid mechanism using a single potential function on $SO(3)\times \mathbb{R}$. The point $(R^*,0)$ denotes  one of the undesired critical points of $U(R,\theta)$. The dashed and solid black curves represent the discrete jumps and  continuous flows of $(R,\theta)$, respectively. 
		}
		\label{fig:diagram1}
	\end{figure}
	
	Once a nonempty finite set  $\Theta \subset \mathbb{R}$ and a potential function $U$ satisfying the basic Assumption \ref{assum:1} are constructed, as it will be shown later, the flow and jump dynamics governing the evolution of $\theta$ and in turn of $U$ will be designed to avoid the undesired critical points, leaving $\mathcal{A}_o$ as the unique attractor. In fact, $\theta$ flows when the state $(R,\theta)$ is away from the set $\Psi_U\setminus\{\mathcal{A}_o\}$, and jumps  to some $ {\theta}'\in \Theta$ leading to minimum value of $U(R,\theta')$) when the state $(R,\theta)$ is in the neighborhood of the set $\Psi_U\setminus\{\mathcal{A}_o\}$. We propose the following hybrid dynamics $\mathcal{H}_\theta$ for $\theta$:
	\begin{align}
		\mathcal{H}_\theta: \begin{cases}
			\dot{\theta} ~~ = f(R,\theta), & (R,\theta)\in \mathcal{F}\\
			\theta^+ \in g(R, \theta), & (R,\theta)\in \mathcal{J}
		\end{cases}\label{eqn:dynamics_theta}
	\end{align}
	with $\theta(0,0)\in  \mathbb{R}$. The flow and jump sets are defined as 	
	\begin{subequations}\label{eqn:Fset-Jset} 
		\begin{align}
			\mathcal{F}&:= \left\{(R,\theta)\in SO(3) \times \mathbb{R}:  \mu_U(R,\theta)\leq \delta \right\} \label{eqn:Fset} \\
			\mathcal{J}&:= \left\{(R,\theta) \in SO(3)\times \mathbb{R}:  \mu_U(R,\theta)\geq \delta \right\} \label{eqn:Jset}
		\end{align} 
	\end{subequations}
	with  $\mu_U$ defined in \eqref{eqn:defu_U}, and the flow map  $f:SO(3)\times \mathbb{R} \to \mathbb{R}$ and jump map   $g:SO(3)\times \mathbb{R} \rightrightarrows \Theta$  are  defined as 
	\begin{align}
		f(R,\theta)&:= -  k_\theta\nabla_{\theta} U(R,\theta) \label{eqn:ftheta} \\
		g(R,\theta) &:= \left\{\theta \in \Theta :  \theta= \arg \min\nolimits_{{\theta}'\in \Theta}U(R,{\theta}') \right\}  \label{eqn:gtheta}
	\end{align} 
	with constant scalar $k_\theta>0$. By Assumption \ref{assum:1},  the design of the jump set $\mathcal{J}$  implies that all the undesired critical points are located in the jump set, \ie, $\Psi_U\setminus\{\mathcal{A}_o\} \subset \mathcal{J}$. Note that the flow map $f$ in \eqref{eqn:ftheta} is nothing else but the negative gradient of $U$ with respect to $\theta$ contributing to driving the state $(R,\theta)$ towards the critical points of $U$. The jump map $g$ in \eqref{eqn:gtheta} is designed to drive (through jumps) the state $(R,\theta)$ away from the undesired critical points.
	From the definitions of the jump set $\mathcal{J}$ and jump map $g$, it is clear that every $(R,\theta)\in \mathcal{J}$ , one has $U(R,\theta)-U(R,g(R,\theta)) = U(R,\theta)-\min\nolimits_{ {\theta}'\in \Theta}U(R,{\theta}') = \mu_U(R,\theta) \geq \delta$, which guarantees a minimum decrease of the potential function $U$ by  $\delta$, after each jump.

	\section{Hybrid Feedback Design} \label{sec:basic} 
	We propose the following hybrid feedback tracking control scheme: 
	\begin{align}
		\underbrace{
			\begin{array}{l}
				\tau = \Upsilon(R_e, \omega_r,z)   -\kappa(R_e,\theta,\omega_e)\\
				\dot{\theta} = f(R_e,\theta)
			\end{array}
		}_{(R_e,\theta)\in \mathcal{F}}
		\underbrace{
			\begin{array}{l}	
				\\
				\\[-0.4cm] 	
				\theta^+  \in g(R_e, \theta)
			\end{array}
		}_{(R_e,\theta)\in \mathcal{J}}  \label{eqn:hybrid_feedback1}
	\end{align}
	where function $\Upsilon$ is defined in \eqref{eqn:Upsilon_def}, the flow and jump sets $\mathcal{F}$ and $\mathcal{J}$  are defined in  \eqref{eqn:Fset} and \eqref{eqn:Jset}, respectively, the function  $\kappa$ is given by
	\begin{align} 	 
		\kappa(R_e,\theta,\omega_e)&:= 2 k_R\psi(R_e\T \nabla_{R_e} U(R_e,\theta)) + k_\omega  \omega_e  \label{eqn:kappa}  
	\end{align}
	with constant scalars $k_R, k_\omega>0$, and the maps $f$ and $g$ (in terms of $R_e$ and $\theta$) are defined in \eqref{eqn:ftheta} and \eqref{eqn:gtheta}, respectively. The main difference between the proposed hybrid feedback \eqref{eqn:hybrid_feedback1}-\eqref{eqn:kappa} with respect to the ones proposed in \cite{mayhew2013synergistic,berkane2018hybrid} is the extended hybrid dynamics of the auxiliary variable $\theta$ which modifies (continuously) the potential function in the flow set $\mathcal{F}$ and modifies (through jumps) the potential function in the jump set $\mathcal{J}$ (\ie, in the neighborhood of the undesired critical points of $U$). The gradient of the potential function $U$, with the extended hybrid dynamics of $\theta$, is used in the control to force  $\mathcal{A}_o$ to be a global attractor. Fig. \ref{fig:diagram3} illustrates the proposed hybrid feedback strategy.

	\begin{figure}[!ht]
		\centering
		\includegraphics[width=0.94\linewidth]{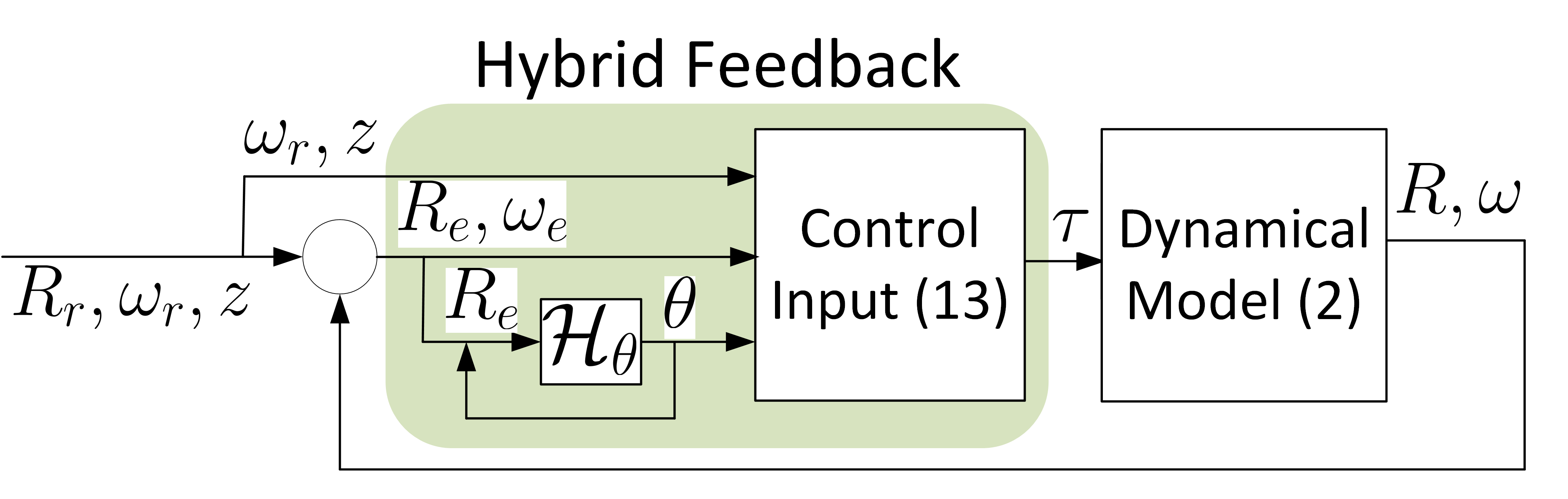}
		\caption{Hybrid feedback strategy for global attitude tracking.}
		\label{fig:diagram3}
	\end{figure}

	Define the new state space $\mathcal{S}:=SO(3) \times \mathbb{R} \times \mathbb{R}^3 \times \mathcal{W}_d $  and the new state $x:=(R_e,\theta,\omega_e,R_r,\omega_r) \in \mathcal{S}$. In view of \eqref{eqn:Re-Jomega_e}, \eqref{eqn:dynamics_theta}-\eqref{eqn:Fset-Jset} and \eqref{eqn:hybrid_feedback1}, one has the following hybrid closed-loop system:
	\begin{equation}
		\begin{cases}
			\dot{x}  ~~\in F(x), & x \in \mathcal{F}_c:=\{x\in \mathcal{S}  : (R_e,\theta)\in \mathcal{F}\} \\
			x^+ \in G(x), & x \in \mathcal{J}_c:=\{x\in \mathcal{S}  : (R_e,\theta)\in \mathcal{J}\}
		\end{cases} \label{eqn:closed-loop1}
	\end{equation}
	where the flow and jump maps are given by
	\begin{align}
		F(x)&:= \begin{pmatrix}
			R_e \omega_e^\times \\
			f(R_e,\theta)\\
			J^{-1}(\Sigma(R_e,\omega_e,\omega_r)\omega_e   -\kappa(R_e,\theta,\omega_e))  \\
			R_r \omega_r^\times \\
			m\mathbb{B} 
		\end{pmatrix}  \label{eqn:closed-loop1-F} \\
		G(x)&:= \begin{pmatrix}
			R_e,~
			g(R_e,\theta),~
			\omega_e,~
			R_r,~
			\omega_r 
		\end{pmatrix} \label{eqn:closed-loop1-J} 
	\end{align}
	with $\Sigma,f,g$ and $\kappa$ defined in  \eqref{eqn:Sigma_def}, \eqref{eqn:ftheta}, \eqref{eqn:gtheta} and  \eqref{eqn:kappa}, respectively. One can verify that $\mathcal{F}_c\cup\mathcal{J}_c = \mathcal{S}$, $\mathcal{F}_c$ and $\mathcal{J}_c$ are closed, and the hybrid closed-loop system   \eqref{eqn:closed-loop1} is autonomous and satisfies the hybrid basic conditions \cite[Assumption 6.5]{goebel2012hybrid}.
	Now, one can state one of our main results:
	\begin{thm}\label{them:1}
		Let $k_R, k_\omega,k_\theta>0$ and suppose that Assumption  \ref{assum:1} holds.  
		Then,  the set $\mathcal{A}:= \{x \in \mathcal{S}  : (R_e,\theta)= \mathcal{A}_o, \omega_e=0 \} $ is globally asymptotically stable for the hybrid  system  \eqref{eqn:closed-loop1} and the number of jumps is finite.  
	\end{thm}
	\begin{proof}
		See   Appendix \ref{sec:them1}.	
	\end{proof} 
	\begin{rem}
		As shown in the proof of Theorem \ref{them:1}, Assumption \ref{assum:1} is the key to avoid the undesired equilibrium points and ensure GAS for the closed-loop system \eqref{eqn:closed-loop1}. 
		Without a strict decrease of the potential function over the jumps, the trajectories may converge to a level set containing one of the undesired equilibrium points.
	\end{rem}

	Now, under the following additional assumptions on the potential function $U$, we will show that the proposed hybrid controller achieves exponential stability.
	\begin{assum} \label{assum:2}
		There exist constant  scalars $\alpha_1>\alpha_2>0$ such that 
		\begin{align} 
			\nabla_{(R,\theta)}U(R,\theta) &\leq \alpha_1 U(R,\theta), \quad \forall (R,\theta)\in SO(3)\times \mathbb{R} \label{eqn:assump2_2} \\
			\nabla_{(R,\theta)}U(R,\theta) &\geq \alpha_2 U(R,\theta), \quad \forall (R,\theta)\in \mathcal{F}  \label{eqn:assump2_3}
		\end{align}
		with $\nabla_{(R,\theta)}U(R,\theta) := \|\psi(R\T \nabla_{R} U(R,\theta))\|^2 + |\nabla_{\theta} U(R,\theta)|^2$ and $\mathcal{F}$ defined in \eqref{eqn:Fset}.
	\end{assum}
	
	\begin{assum}\label{assum:3}
		Given the dynamics   \eqref{eqn:Re} and   \eqref{eqn:dynamics_theta}, there exist constants $c_R,c_\theta>0$ such that 
		\begin{multline} 
			\| \dot{\psi}(R_e\T \nabla_{R_e} U(R_e,\theta))  \|  \\
			\leq  c_R \|\omega_e\| + c_\theta k_\theta |\nabla_{\theta} U(R_e,\theta) |, ~\forall (R_e,\theta)\in \mathcal{F} \label{eqn:dpsi} 
		\end{multline}
		with $\mathcal{F}$ defined in \eqref{eqn:Fset}.  
	\end{assum}
	Assumptions \ref{assum:2} and \ref{assum:3} impose some bounds on the gradients of the potential function and the time derivative of $\psi(R_e\T \nabla_{R_e} U(R_e,\theta))$. This assumptions are not very restrictive as it is going to be shown later once we present the construction of the potential function in Section \ref{sec:potential}. \\
	Let 
	$|x|_{\mathcal{A}}^2  
	:= U(R_e,\theta)+\|\omega_e\|^2 $. Since   $U$ is a potential function on $SO(3)\times \mathbb{R}$ with respect to $\mathcal{A}_o$, it follows from the definitions of $\mathcal{S}$ and $\mathcal{A}$ that $|x|_{\mathcal{A}}= 0$ for all $x\in \mathcal{A}$, and  $|x|_{\mathcal{A}}> 0$ for all $x\in \mathcal{S}\setminus\mathcal{A}$.  
	Now, one can state the following result:
	\begin{prop}\label{prop:1}
		Let $k_R, k_\omega,k_\theta>0$ and suppose that Assumption  \ref{assum:1}-\ref{assum:3} hold. 
		Then, for every compact set $\varOmega_c \subset SO(3)\times \mathbb{R} \times \mathbb{R}^3$ and every initial condition $x(0,0)\in \varOmega_c  \times \mathcal{W}_d $,  the number of jumps is finite, and there exist $k,\lambda>0$ such that,  each maximal solution $x$ to the hybrid  system    \eqref{eqn:closed-loop1} satisfies
		\begin{equation}
			|x(t,j)|_{\mathcal{A}}^2  \leq k \exp(-\lambda(t+j))|x(0,0)|_{\mathcal{A}}^2  \label{eqn:phi_exp} 
		\end{equation} 
		for all $(t,j)\in \dom x$
	\end{prop}
	\begin{proof}	
		See Appendix \ref{sec:prop1}.	 
	\end{proof}
	\begin{rem} 
		Proposition \ref{prop:1} shows that the tracking error converges exponentially to the set $\mathcal{A}$ for each initial condition in the compact set $\varOmega_c \subset SO(3)\times \mathbb{R} \times \mathbb{R}^3$ satisfying $\mathcal{A} \subseteq \varOmega_c  \times \mathcal{W}_d $ (Note that $\mathcal{W}_d $ is compact by assumption). It is important to mention that the exponential stability proved in Proposition \ref{prop:1} is referred to as semi-global exponential stability, since the parameters  $k,\lambda>0$ depend on the initial conditions which are restricted to an arbitrary compact subset $\varOmega_c  \times \mathcal{W}_d \subset \mathcal{S}$. Since the number of jumps is finite, the hybrid exponential stability can be viewed as the exponential stability in the classical sense (exponential convergence over time). This situation is sometimes referred to as exponentially stability in the $t$-direction \cite{casau2019robust}.
	\end{rem}

	\section{Hybrid Feedback With Torque Smoothing Mechanism} \label{sec:smooth} 
	In order to remove the discontinuities in the control input $\tau$ (caused by the discrete jumps of $\theta$), we propose the following modified hybrid feedback tracking   scheme: 
	\begin{align}
		\underbrace{
			\begin{array}{l}
				\tau = \Upsilon(R_e, \omega_r,z)  - 2k_R \zeta - k_\omega \omega_e\\
				\dot{\theta} = f(R_e,\theta) \\ 
				\dot{\zeta} = h(R_e,\theta,\zeta)
			\end{array}
		}_{(R_e,\theta,\zeta)\in  \widehat{\mathcal{F}}}
		\underbrace{
			\begin{array}{l}	
				\\
				\\[-0.4cm] 	
				\theta^+  \in {g}(R_e, \theta) \\ 
				\zeta^+ = \zeta
			\end{array}
		}_{(R_e,\theta,\zeta)\in  \widehat{\mathcal{J}}}  \label{eqn:hybrid_feedback2}
	\end{align}
	where $\theta(0,0)\in \mathbb{R}, \zeta(0)\in \mathbb{R}^3$, constants $k_R,k_\omega>0$, the maps $f$ and $g$ are defined in \eqref{eqn:ftheta} and \eqref{eqn:gtheta}, the function $\Upsilon$ is defined in \eqref{eqn:Upsilon_def}, and the function  $h$ is given by
	\begin{align} 
		h(R_e,\theta,\zeta) & =    -    k_\zeta (\zeta-  \psi(R_e\T \nabla_{R_e} U(R_e,\theta)))   \label{eqn:h_def} 
	\end{align}
	with constant   $k_\zeta>0$. 
		In this case, the flow and jump sets are given by
		\begin{align}
			\widehat{\mathcal{F}} &:=  \{(R_e, \theta,\zeta) \in \mho:  \mu_W(R_e,\theta,\zeta)\leq \delta'   \} \label{eqn:Fset2} \\
			\widehat{\mathcal{J}} &:=  \{(R_e, \theta,\zeta) \in \mho:  \mu_W(R_e,\theta,\zeta)\geq \delta'  \}  \label{eqn:Jset2}
		\end{align}
		where   $0<\delta'<\delta$, $\mho: =  SO(3)\times \mathbb{R} \times \mathbb{R}^3$ and
		\begin{align}
			\mu_W(R_e,\theta,\zeta) &: = W(R_e,\theta,\zeta) - \min_{{\theta}'\in \Theta}W(R_e,\theta',\zeta) \label{eqn:defu_W} \\
			W(R_e,\theta,\zeta)  &:= U(R_e,\theta) +   {\varrho}\|\zeta- \psi(R_e\T \nabla_{R_e} U(R_e,\theta))\|^2\label{eqn:defW}
		\end{align}
		with some $\varrho>0$ to be designed later. 
	The main difference between this hybrid control scheme and the previous one in \eqref{eqn:hybrid_feedback1}, is the use of an dynamical variable $\zeta$ that bears the hybrid jumps of $\theta$ resulting in a jump-free control signal. As shown in \eqref{eqn:hybrid_feedback2}-\eqref{eqn:h_def}, the dynamics of $\zeta$ allow to relocate the jumps one integrator away from the control torque. This mechanism leads to a continuous torque input since $\zeta$ is continuous (not necessary differentiable due to the discrete jumps of $\theta$ in the gradient-based term).  
	
	Since $U$ is a potential function on $SO(3)\times \mathbb{R}$ with respect to $\mathcal{A}_o$ and $\psi(R_e\T \nabla_{R_e} U(R_e,\theta))=0$ as $(R_e,\theta)= \mathcal{A}_o$, one can show that   $W$ defined in \eqref{eqn:defW} is a potential function on $\mho$ with respect to $\widehat{\mathcal{A}}_o:=(R_e=I_3,\theta=0,\zeta=0)$, and the set of its critical points is given by $\Psi_{W}:=\{(R_e,\theta,\zeta)\in \mho: (R_e,\theta)\in \Psi_{U}, \zeta = 0 \}$. To ensure that all the undesired critical points of $W$ are located in the jump set $\widehat{\mathcal{J}}$ in \eqref{eqn:Jset2}, we consider the following assumption:
	\begin{assum}\label{assum:4}
		There exists a constant $c_\psi>0$ such that $\|\psi(R_e\T \nabla_{R_e} U(R_e,\theta))\|\leq c_\psi$ for all $(R_e,\theta)\in SO(3)\times \mathbb{R}$.
	\end{assum} 
	\begin{lem} \label{lem:W}
		Let  Assumption \ref{assum:1} and \ref{assum:4} hold, and $0<\delta'<\delta$ and  $0<\varrho<  {(\delta-\delta')}/{c_\psi^2}$, then the following inequality holds:
		\begin{align}
			\mu_W(R_e,\theta,\zeta)  >  \delta',   \quad \forall (R_e,\theta,\zeta)\in \Psi_W \setminus \{\widehat{\mathcal{A}}_o\} \label{eqn:mu_W_delta}.
		\end{align} 
	\end{lem}
	The proof of Lemma \ref{lem:W} is given in  Appendix \ref{sec:lemW}. From the definition of $\widehat{\mathcal{J}}$ in \eqref{eqn:Jset2}, Lemma \ref{lem:W} implies that all the undesired critical points of $W$ in \eqref{eqn:defW} are located in the jump set $\widehat{\mathcal{J}}$   (\ie, $\Psi_W\setminus\{\widehat{\mathcal{A}}_o\} \subset \widehat{\mathcal{J}}$) under Assumption \ref{assum:1}, \ref{assum:4} and some  small enough positive constant $\varrho$.

	Define the new state space  $\widehat{\mathcal{S}}:=\mathcal{S} \times   \mathbb{R}^3 $  and the new state $\hat{x}:=(x,\zeta) \in \widehat{\mathcal{S}} $. In view of \eqref{eqn:Re-Jomega_e}, \eqref{eqn:dynamics_theta}-\eqref{eqn:Fset-Jset} and \eqref{eqn:hybrid_feedback2}-\eqref{eqn:h_def}, one has the following hybrid closed-loop system:
	\begin{equation}  
		\begin{cases}
			\dot{\hat{x}}  ~~\in \widehat{F}(\hat{x}), & \hat{x} \in  \widehat{\mathcal{F}}_c:=\{\hat{x}\in \widehat{\mathcal{S}}  : (R_e,\theta,\zeta)\in \widehat{\mathcal{F}}\} \\
			\hat{x}^+ \in \widehat{G}(\hat{x}), & \hat{x} \in  \widehat{\mathcal{J}}_c:=\{\hat{x}\in \widehat{\mathcal{S}}  : (R_e,\theta,\zeta)\in \widehat{\mathcal{J}}\}
		\end{cases}  \label{eqn:closed-loop2}
	\end{equation}
	where the flow and jump maps are given by
	\begin{align}
		\widehat{F}(\hat{x})&:= \begin{pmatrix}
			R_e \omega_e^\times  \\
			f(R_e,\theta) \\
			J^{-1}(\Sigma(R_e,\omega_e,\omega_r)\omega_e   -2k_R \zeta - k_\omega \omega_e)   \\
			R_r \omega_r^\times \\
			m\mathbb{B} \\		
			-    k_\zeta (\zeta-  \psi(R_e\T \nabla_{R_e} U(R_e,\theta))) 
		\end{pmatrix}  \label{eqn:hatF}\\
		\widehat{G}(\hat{x})&:= \begin{pmatrix}
			R_e, ~
			{g}(R_e,\theta),  ~
			\omega_e,~
			R_r,~
			\omega_r, ~
			\zeta 
		\end{pmatrix}
	\end{align}
	with $\Sigma, f$ and $g$ defined in  \eqref{eqn:Sigma_def}, \eqref{eqn:ftheta} and  \eqref{eqn:gtheta}, respectively. One can verify that $\widehat{\mathcal{F}}_c\cup \widehat{\mathcal{J}_c} = \widehat{\mathcal{S}}$, $\widehat{\mathcal{F}}_c$ and $\widehat{\mathcal{J}}_c$ are closed, and the hybrid closed-loop system \eqref{eqn:closed-loop2} satisfies the hybrid basic conditions \cite[Assumption 6.5]{goebel2012hybrid}.  
	The properties of the set $\widehat{\mathcal{A}}:= \{\hat{x} \in \widehat{\mathcal{S}}  : (R _e,\theta,\zeta)=  \widehat{\mathcal{A}}_o,  \omega_e=0 \} $ for the closed-loop system \eqref{eqn:closed-loop2} are stated in the following theorem:
	\begin{thm}\label{them:2} 
		Let $k_R, k_\omega,k_\theta >0$, and suppose that Assumption \ref{assum:1}, \ref{assum:3} and \ref{assum:4} hold.
		Then, there exist constants $k^*_\zeta>0$ and $0<\delta'<\delta$ such that, for every $k_\zeta>k_\zeta^*$ and   $0<\varrho<  {(\delta-\delta')}/{c_\psi^2}$, the set $\widehat{\mathcal{A}}$ is globally asymptotically stable for the hybrid system  \eqref{eqn:closed-loop2}   and the number of jumps is finite. 
	\end{thm}
	\begin{proof}
		See   Appendix \ref{sec:them2}.		
	\end{proof} 
	Following similar steps as in the proof of Proposition \ref{prop:1}, one can also show that, under the additional Assumption   \ref{assum:2}, the proposed hybrid feedback, with the torque smoothing mechanism, guarantees semi-global exponential stability. Note that the high gain condition on $k_\zeta$ in Theorem \ref{them:2}  can be relaxed by considering the following dynamics for $\zeta$:
	\begin{align} 
		\dot{\zeta}  &=   \varphi(\hat{x}) -    k_\zeta \left( \zeta- \psi(R_e\T \nabla_{R_e} U(R_e,\theta))\right)   \label{eqn:h_def2} 
	\end{align}
	where  $\varphi(\hat{x})  := \dot{\psi}(R_e\T \nabla_{R_e} U(R_e,\theta)) +  \frac{1}{\varrho}  \omega_e$ with some constants $k_\zeta>0$ and $0<\varrho< {(\delta-\delta')}/{c_\psi^2}$. With this modification, global asymptotic stability is also guaranteed as in Theorem \ref{them:2}, and the proof is omitted here.

	\section{Hybrid Feedback Without Velocity Measurements}\label{sec:velocity_free}
	Inspired by the work in \cite{TayebiCDC2006,tayebi2008unit,berkane2018hybrid}, we propose a new hybrid feedback for global attitude tracking without using the measurements of angular velocity $\omega$. In practice, obviating the need of the angular velocity measurements is of great interest in applications relying on expensive and prone-to-failure gyroscopes. In the case where gyroscopes are available, this velocity-free tracking controller can also be used as a backup scheme triggered by gyro failure. 
	
	Consider the auxiliary state $(\bar{R},\bar{\theta}) \in SO(3)\times \mathbb{R}$ and the following   hybrid auxiliary system:
	\begin{align}
		&\underbrace{
			\begin{array}{l}
				\dot{\bar{R}} = \bar{R} (\tilde{R} \beta)^\times \\
				\dot{\bar{\theta}} ~ = f(\tilde{R},\bar{\theta})
			\end{array}
		}_{(\tilde{R},\bar{\theta}) \in \mathcal{F}} ~ 
		\underbrace{
			\begin{array}{l}
				\bar{R}^+ = \bar{R}  \\
				\bar{\theta}^+ ~~ \in~ g(\tilde{R},\bar{\theta})
			\end{array}
		}_{(\tilde{R},\bar{\theta}) \in \mathcal{J}} \label{eqn:bar_R} 
	\end{align}
	where $\bar{R}(0)\in SO(3), \bar{\theta}(0,0)\in \mathbb{R}$, $\tilde{R} = \bar{R}\T R_e$, the flow and jump sets $\mathcal{F}$ and $\mathcal{J}$  are defined in  \eqref{eqn:Fset} and \eqref{eqn:Jset}, respectively, and   $\beta$ is given by
	\begin{equation}
		\beta  =  \Gamma \psi(\tilde{R}\T \nabla_{\tilde{R}} U(\tilde{R},\bar{\theta})) \label{eqn:beta}
	\end{equation}
	with a symmetric positive definite matrix $\Gamma$. The dynamics of the auxiliary variable $\bar{R}$ are inspired from \cite{tayebi2008unit,berkane2018hybrid}, and the maps $f$ and $g$ are given in \eqref{eqn:ftheta} and \eqref{eqn:gtheta}, respectively. 
	
	We propose the following velocity-free hybrid feedback tracking   scheme: 
	\begin{align}
		\underbrace{
			\begin{array}{l}
				\tau = \Upsilon(R_e, \omega_r,z)  - \bar{\kappa}(R_e,\theta,\tilde{R},\bar{\theta})\\
				\dot{\theta} = f(R_e,\theta)  \\
			\end{array}
		}_{(R_e,\theta)\in  {\mathcal{F}}}
		\underbrace{
			\begin{array}{l}	
				\\
				\\[-0.4cm] 	
				\theta^+  \in {g}(R_e, \theta)  
			\end{array}
		}_{(R_e,\theta)\in  {\mathcal{J}}}  \label{eqn:hybrid_feedback3}
	\end{align}
	where the hybrid dynamics of the auxiliary state $(\bar{R},\bar{\theta})$ are given in \eqref{eqn:bar_R}, and the function  $\bar{\kappa}$ is given by
	\begin{multline} 	 
		\bar{\kappa}(R_e,\theta,\tilde{R},\bar{\theta}) := 2 k_R\psi(R_e\T \nabla_{R_e} U(R_e,\theta))   \\
		+  2k_\beta \psi(\tilde{R}\T \nabla_{\tilde{R}} U(\tilde{R},\bar{\theta}))  \label{eqn:kappa2}  
	\end{multline}
	with constants $k_R,k_\beta>0$. The flow and jump sets $\mathcal{F}$ and $\mathcal{J}$  are defined in  \eqref{eqn:Fset} and \eqref{eqn:Jset}, respectively.

	Instead of using the angular velocity tracking error $\omega_e$ as in the hybrid controllers \eqref{eqn:hybrid_feedback1} and \eqref{eqn:hybrid_feedback2}, a new term  generated from the gradient of $U(\tilde{R},\bar{\theta})$  is considered in the design of the control input $\tau$ in \eqref{eqn:hybrid_feedback3}. This term, relying on the output of the auxiliary system \eqref{eqn:bar_R}-\eqref{eqn:beta}, allows to generate the necessary damping in the absence of the angular velocity measurements. In fact, an appropriate design of the input  $\beta$ of the auxiliary system, ensures the convergence of $\beta$ to $\omega_e$ as $\tilde{R} \rightarrow I_3$, which consequently leads to $R_e \rightarrow I_3$ and $\omega_e \rightarrow 0$. 
	Fig. \ref{fig:diagram4} illustrates the proposed velocity-free hybrid feedback strategy.
	
	\begin{figure}[!ht]
		\centering
		\includegraphics[width=0.9\linewidth]{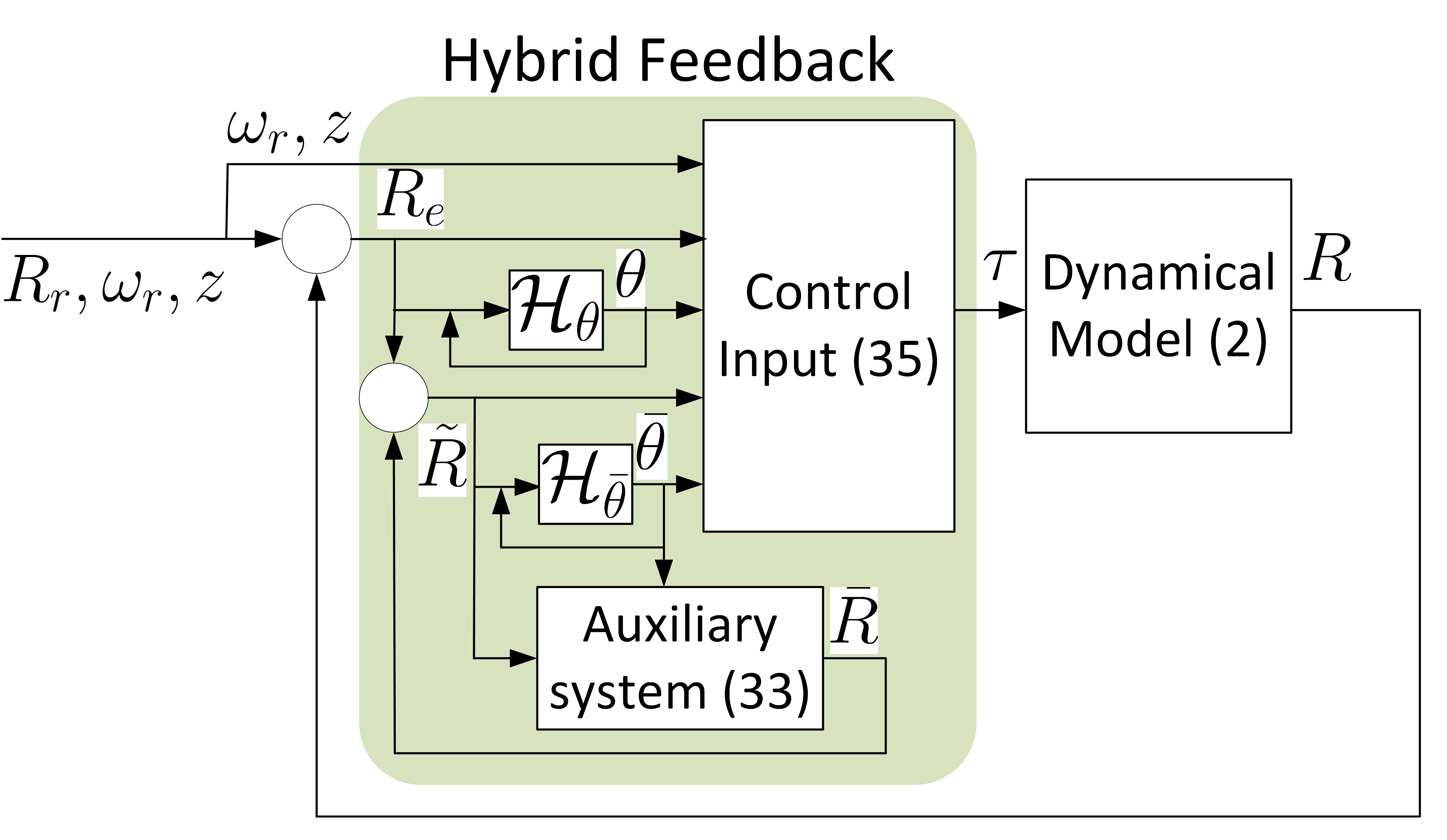}
		\caption{Velocity-free hybrid feedback strategy for global attitude tracking using a hybrid auxiliary system.}
		\label{fig:diagram4}
	\end{figure}
	
	Define the new state space  $\overline{\mathcal{S}}:=\mathcal{S} \times SO(3) \times   \mathbb{R} $  and the new state $\bar{x}:=(x,\tilde{R},\bar{\theta}) \in \overline{\mathcal{S}} $. In view of \eqref{eqn:Re-Jomega_e}, \eqref{eqn:dynamics_theta}-\eqref{eqn:Fset-Jset} and \eqref{eqn:bar_R}-\eqref{eqn:hybrid_feedback3}, one has the following closed-loop system: 
	\begin{equation}
		\begin{cases}
			\dot{\bar{x}}  ~~\in \overline{F}(\bar{x}), & \bar{x} \in  \overline{\mathcal{F}}_c:= \mathcal{F}_c \times \mathcal{F} \\
			\bar{x}^+ \in \overline{G}(\bar{x}), & \bar{x} \in  \overline{\mathcal{J}}_c:= \overline{\mathcal{J}}_{c1} \cup \overline{\mathcal{J}}_{c2}
		\end{cases} \label{eqn:closed-loop3}
	\end{equation}
	where $\overline{\mathcal{J}}_{c1}:= \mathcal{J}_c \times SO(3) \times \mathbb{R}$ and $ \overline{\mathcal{J}}_{c2}:= \mathcal{S}  \times \mathcal{J}$, and
	the flow and jump maps are given by
	\begin{align}
		\overline{F}(\hat{x})&:= \begin{pmatrix}
			R_e \omega_e^\times \\
			f(R_e,\theta)\\
			J^{-1}(\Sigma(R_e,\omega_e,\omega_r)\omega_e   - \bar{\kappa}(R_e,\theta,\tilde{R},\bar{\theta}))\\
			R_r \omega_r^\times \\
			m\mathbb{B} \\		
			\tilde{R}(\omega_e - \Gamma \psi(\tilde{R}\T \nabla_{\tilde{R}} U(\tilde{R},\bar{\theta})) )^\times \\
			f(\tilde{R},\bar{\theta})
		\end{pmatrix}  \label{eqn:barF}\\
		\overline{G}(\bar{x})&:= \begin{pmatrix}
			R_e, 
			g_\theta(R_e,\theta),  
			\omega_e, 
			R_r, 
			\omega_r, 
			\tilde{R}, 
			g_{\bar{\theta}}(\tilde{R},\bar{\theta}) 
		\end{pmatrix}
	\end{align}
	with   $\Sigma, f$ and $g$   defined in  \eqref{eqn:Sigma_def}, \eqref{eqn:ftheta} and \eqref{eqn:gtheta}, respectively. The function  $g_\theta(R_e,\theta)$ is defined as: $g_\theta(R_e,\theta) = g(R_e,\theta)$ if $(R_e,\theta)\in \mathcal{J}$ (\ie, $\bar{x}\in \overline{\mathcal{J}}_{c1}$) otherwise $g_\theta(R_e,\theta) =  \theta $, and the function $g_{\bar{\theta}}(\tilde{R},\bar{\theta})$ is defined as: $g_{\bar{\theta}}(\tilde{R},\bar{\theta}) = g(\tilde{R},\bar{\theta})$ if $(\tilde{R},\bar{\theta})\in \mathcal{J}$ (\ie, $\bar{x}\in \overline{\mathcal{J}}_{c2}$) otherwise $g_{\bar{\theta}}(\tilde{R},\bar{\theta}) =  \bar{\theta} $. One can verify that $\overline{\mathcal{F}}_c\cup \overline{\mathcal{J}_c} = \overline{\mathcal{S}}$, sets $\overline{\mathcal{F}}_c$ and $\overline{\mathcal{J}}_c$ are closed, and the hybrid  system   \eqref{eqn:closed-loop3} satisfies the hybrid basic conditions \cite[Assumption 6.5]{goebel2012hybrid}. 
	Now, one can state the following  result:
	\begin{thm}\label{them:3}
		Let $k_R, k_\beta,k_\theta>0$, $\Gamma=\Gamma\T$ be positive definite, and suppose that Assumption  \ref{assum:1} holds.  
		Then,  the set $\overline{\mathcal{A}}:= \{\bar{x} \in \overline{\mathcal{S}} : ( R _e,\theta)= (\tilde{R},\bar{\theta})= \mathcal{A}_o, \omega_e=0 \} $ is globally asymptotically stable for the hybrid  system \eqref{eqn:closed-loop3} and the number of jumps is finite.  
	\end{thm}
	\begin{proof}
		See   Appendix \ref{sec:them3}.	
	\end{proof}
	\begin{rem} 
		Following similar steps as in the proof of Proposition \ref{prop:1}, one can also show that, under the additional Assumptions   \ref{assum:2} and \ref{assum:3}, the proposed velocity-free hybrid tracking controller guarantees semi-global exponential stability. Moreover, similar to Section \ref{sec:smooth}, the proposed velocity-free hybrid attitude tracking controller \eqref{eqn:hybrid_feedback3} can be further extended with a torque smoothing mechanism by filtering the terms $\psi(R_e\T \nabla_{R_e} U(R_e,\theta))$ and $\psi(\tilde{R}\T \nabla_{\tilde{R}} U(\tilde{R},\bar{\theta}))$  as in Section \ref{sec:smooth}
		to obtain a jump-free torque input.  
	\end{rem}

	\section{Construction of The Potential Function on $SO(3)\times \mathbb{R}$} \label{sec:potential}
	Our proposed designs in the previous sections rely on the existence of a potential function $U$  on $SO(3)\times \mathbb{R}$  with respect to $\mathcal{A}_o$, satisfying Assumptions \ref{assum:1}-\ref{assum:4}.  In this  section, we will provide  a systematic procedure for the construction of such potential function using the angular warping techniques inspired by \cite{mayhew2011synergistic}.
	
	Consider the following   transformation map  $\mathcal{T}: SO(3)\times \mathbb{R} \to SO(3)$ :
	\begin{equation}
		\mathcal{T}(R,\theta) =  R \mathcal{R}_a(\theta,u)   \label{eqn:transfT}
	\end{equation}
	where $R\in SO(3)$, $u\in \mathbb{S}^2$ is a constant unit vector and  $\theta \in \mathbb{R}$ is a real-valued variable with hybrid dynamics specified in Section \ref{sec:min-resetting}. From \eqref{eqn:transfT}, $\mathcal{T}$ applies a rotation by an angle $\theta$ to $R$ about the unit vector $u$. The main difference compared to the transformation maps considered in \cite{mayhew2011synergistic,mayhew2011hybrid,mayhew2013synergistic,berkane2017construction,berkane2017hybrid}, is that the angular warping angle $\theta$ considered in \eqref{eqn:transfT} is an independent real-valued variable with hybrid flows and jumps.  
	
	Consider a modified trace function   $V(R) = \tr(A(I_3-R))$ with $A=A\T$ being a positive definite matrix. It follows from \cite{mayhew2013synergistic,mahony2008nonlinear} that the set of all the critical points of $V(R)$ is given by $\Psi_V  = \{I_3\} \cup \mathcal{R}_a(\pi,\mathcal{E}(A)) $ with $\mathcal{E}(A)$ denoting the set of unit eigenvectors of $A$. 
	Let us introduce the following real-valued function on $SO(3)\times \mathbb{R}$ as
	\begin{align}
		U(R,\theta)
		&=\tr(A(I-\mathcal{T}(R,\theta))) + \frac{\gamma }{2  } \theta^2    \label{eqn:URtheta}	
	\end{align}
	with some constant $\gamma >0$ to be designed. The first term of $U$ is modified from the potential function $V$ inspired by \cite{mayhew2011hybrid,mayhew2011synergistic}, and the second term is a quadratic term in $\theta$.  From the definition of  $\mathcal{T}$ in \eqref{eqn:transfT}, one can easily verify that $U(R,\theta)\geq 0$ for all $(R,\theta)\in SO(3)\times \mathbb{R}$, and $U(R,\theta) = 0$ if and only if $(R=I_3,\theta=0)$. Hence, $U$ is a potential function on $SO(3)\times \mathbb{R}$ with respect to $\mathcal{A}_o$.
	The following lemma provides useful properties of the potential function $U$.
	\begin{lem}\label{lemma:U}
		Let   $A=A\T$ be a  positive definite matrix.   Consider  the potential function $U$ defined in \eqref{eqn:URtheta}, and the trajectories generated by $\dot{R}=R\omega^\times$ and $\dot{\theta}=\nu$ with $R(0)\in SO(3), \theta(0)\in \mathbb{R}, \omega\in \mathbb{R}^3, \nu\in \mathbb{R}$. Then, for all $(R,\theta)\in SO(3)\times \mathbb{R}$ the following statements hold:
		\begin{subequations}
			\begin{align}
				\dot{\mathcal{T}}(R,\theta)
				& = \mathcal{T}(R,\theta) (\mathcal{R}_a(\theta,u)\T \omega + \nu u)^\times\label{eqn:dTRtheta} \\
				\psi (R\T \nabla_{R} U(R,\theta)) &=  \mathcal{R}_a(\theta,u) \psi(A\mathcal{T}(R,\theta)) \label{eqn:gradient_R_U}  \\
				\nabla_{\theta} U(R,\theta) & = \gamma \theta + 2u \T  \psi (A \mathcal{T}(R,\theta)) \label{eqn:gradient_theta_U} \\
				\mathcal{A}_o & \in \Psi_{U}  := \Psi_{V} \times \{0\} \label{eqn:crit_U_R} \\ 
				\dot{\psi}(R\T \nabla_{R} U(R,\theta))  
				&   =  \mathcal{D}_{R} (R,\theta) \omega + \mathcal{D}_{\theta} (R,\theta) \nu \label{eqn:dot_psi}
			\end{align}
		\end{subequations}
		where  $	\mathcal{D}_{R} (R,\theta)  := \mathcal{R}_a(\theta,u)E(A\mathcal{T}(R,\theta))   \mathcal{R}_a\T(\theta,u) \in \mathbb{R}^{3\times 3} $ with $E(A) := \frac{1}{2}(\tr(A)I_3-A\T), \forall A\in \mathbb{R}^{3\times 3}$, and $\mathcal{D}_{\theta} (R,\theta) := \mathcal{R}_a(\theta,u)E(A\mathcal{T}(R,\theta)) u  - (\mathcal{R}_a(\theta,u) \psi(A\mathcal{T}(R,\theta)) ) ^\times  u  \in \mathbb{R}^3$.

	\end{lem}
	\begin{proof}
		See  Appendix \ref{sec:lemma_U}.
	\end{proof} 
	\begin{rem}
		Note that in \eqref{eqn:dTRtheta}, we obtain a different form of the time derivative of the transformation map  $\mathcal{T}$ on $SO(3)\times \mathbb{R}$  compared to \cite[Theorem 6]{mayhew2011synergistic} and \cite[Lemma 1]{berkane2017construction}. As mentioned before,  the transformation map $\mathcal{T}$ in \cite{mayhew2011synergistic} and \cite{berkane2017construction} needs to be a (local) diffeomorphism to obtain the new set of critical points of the potential functions after transformation. In our approach, the set of critical points of the potential function $U$ on $SO(3)\times \mathbb{R}$ with respect to $\mathcal{A}_o$, denoted by $\Psi_U$ in \eqref{eqn:crit_U_R}, can be easily obtained from \eqref{eqn:gradient_R_U} and \eqref{eqn:gradient_theta_U} without any additional conditions. Moreover, from \eqref{eqn:crit_U_R},  the set $\Psi_U$ is given by a simple extension of the set $\Psi_V$ 
		with $\theta\in \{0\}$. This property allows us to set the state $(R,\theta)$ away from the undesired critical points in $\Psi_U\setminus\{\mathcal{A}_o\}$ by resetting the variable $\theta$ to some non-zero values, which is the key of our reset mechanism proposed in Section \ref{sec:min-resetting}.
	\end{rem}

	We define the set of parameters $\mathcal{P}_U:=\{\Theta,A,u,\gamma,\delta\}$ with  a finite non-empty real-valued set $\Theta \subset \mathbb{R}$, a matrix $A=A\T \in \mathbb{R}^{3\times 3}$, a unit vector $u\in \mathbb{S}^2$, and constant scalars $\gamma,\delta>0$. The following proposition verifies all the conditions in Assumption \ref{assum:1}-\ref{assum:4} required by Section  \ref{sec:basic}-\ref{sec:velocity_free}.  
	\begin{prop}\label{prop:3}
		Consider the potential function $U$ defined in \eqref{eqn:URtheta}. Then, Assumptions \ref{assum:2}-\ref{assum:4} hold for any $\gamma>0,u\in \mathbb{S}^2$ and $A=A\T$ being positive definite.
		Moreover, the basic Assumption \ref{assum:1} holds given the set $\mathcal{P}_U$ defined as follows:
		\begin{align}
			\mathcal{P}_U:
			\left\{
			\begin{array}{l}
				\Theta = \{|\theta_i|  \in (0, \pi],i=1,\cdots,m\} \\
				A=A\T: 0 < \lambda_1^A \leq \lambda_2^A < \lambda_3^A \\
				u = \alpha_1 v_1^A+\alpha_2 v_2^A + \alpha_3 v_3^A \in \mathbb{S}^2  \\
				\gamma < \frac{4\Delta^*}{\pi^2} \\
				\delta <  ( \frac{4\Delta^* }{\pi^2} -\gamma)  \frac{\theta_M^2}{2},\ \theta_M  := \sup_{\theta'\in \Theta} |\theta'|
			\end{array}
			\right. \label{eqn:PU}
		\end{align}
		where  scalars $\alpha_1,\alpha_2, \alpha_3$ and $\Delta^*$ are given as per one of the following three cases:
		\begin{itemize}
			\item[1)]  if $ \lambda_1^A = \lambda_2^A$, 
			$\alpha_3^2 = 1-\frac{\lambda_2^A}{\lambda_3^A}$ and $ \Delta^* = \lambda_1^A(1-\frac{\lambda_2^A}{\lambda_3^A})$.	
			\item[2)]   if  $\lambda_2^A \geq \frac{\lambda_1^A \lambda_3^A}{\lambda^A_3 - \lambda_1^A}$,  
			$
			\alpha_i^2 = \frac{\lambda_i^A}{\lambda_2^A +\lambda_3^A},  i\in \{2,3\}$ and $\Delta^* = \lambda_1^A$.
			\item[3)]  if  $ \lambda_1^A <  \lambda_2^A < \frac{\lambda_1^A \lambda_3^A}{\lambda^A_3 - \lambda_1^A}$,  
			$
			\alpha_i^2 = 1-  \frac{4\prod_{j\neq i} \lambda_j^A}{(\sum_{l=1}^3\sum_{k\neq l}^3\lambda_l^A\lambda^A_k)}, i\in \{1,2,3\}$ and $
			\Delta^* =  \frac{4\prod_{j} \lambda_j^A}{\sum_{l=1}^3\sum_{k\neq l}^3\lambda_l^A\lambda^A_k}.
			$
		\end{itemize}
		with $(\lambda_i^A,v_i^A)$ denoting the $i$-th pair of eigenvalue-eigenvector of matrix $A$.
	\end{prop}

	\begin{proof}
		See  Appendix \ref{sec:prop3}
	\end{proof}

	\begin{rem}
		As shown in the proof of Proposition \ref{prop:3}, Assumption 1 holds if there exists a unit vector $u\in \mathbb{S}^2$ such that one has $\Delta^* = \min_{v\in \mathcal{E}(A)}\Delta(v,u)>0$, where $\Delta(u,v) = u\T \left(\tr(A)I_3 - A - 2v\T A v(I_3 - vv\T) \right) u$. Proposition \ref{prop:3} provides a design option for the potential function $U$ through the choice of the set $\mathcal{P}_U$ given in \eqref{eqn:PU}. Inspired by \cite{berkane2017construction}, the unit vector $u$ is designed  in terms of the  eigenvalues and eigenvectors of the matrix $A$ with $0 < \lambda_1^A \leq \lambda_2^A < \lambda_3^A$. The choice of the unit vector $u$ in Proposition \ref{prop:3} is optimal in terms of $\sup_{u\in \mathbb{S}^2} (\min_{v\in \mathcal{E}(A)}\Delta(v,u))$    (see \cite[Proposition 2]{berkane2017construction} for more details). 
	\end{rem}
	
	\begin{rem}
		Note that a decrease in the value of $\gamma$ results in an increase of the   gap $\delta$ (strengthening the robustness to measurement noise). However, it may slow down the convergence of $\theta$ as per \eqref{eqn:ftheta}, leading to lower convergence rates for the overall closed-loop system. Hence, the parameter $\gamma$ needs to be carefully chosen via a trade-off between the robustness to measurement noise and the convergence rates of the overall closed-loop system. The performance of the closed-loop system, in terms of convergence rates, with different choices of $\gamma$ is illustrated in the simulation section. 
	\end{rem}
	
	\section{Extension to Global Pose Tracking on $SE(3)$} \label{sec:SE(3)}
		In this section, we extend our previous hybrid control strategy on $SO(3)$ to  the 3-dimensional \textit{Special Euclidean group} $SE(3)$, defined as
		\begin{align*}
			SE(3):= \left\{X = \begin{bmatrix}
				R & p\\
				0 & 1
			\end{bmatrix} \in \mathbb{R}^{4\times 4}: R\in SO(3), p\in \mathbb{R}^3\right\}
		\end{align*}
		with $R$ and $p$ denoting the rotation and position, respectively. The Lie algebra of $SE(3)$, denoted by $\mathfrak{se}(3)$,  is defined as  
		\begin{align*}
			\mathfrak{se}(3):= \left\{U = \begin{bmatrix}
				\omega^\times & v\\
				0 & 0
			\end{bmatrix} \in \mathbb{R}^{4\times 4}: \omega^\times \in \mathfrak{so}(3), v\in \mathbb{R}^3\right\}
		\end{align*}
		with $\omega$ and $v$ denoting the angular and linear velocities, respectively. The definitions of the maps $(\cdot)^\wedge$, $\bar{\psi}$, the adjoint action map $\Ad$ and the adjoint operator  $\ad$, and their properties are given in Appendix \ref{sec:property_SE(3)}.

		We consider the following fully actuated system on $SE(3)$
		\begin{align}
			\begin{cases}
				\dot{X} &= X\xi^\wedge\\
				\mathbb{I} \dot{\xi} &= \ad_\xi\T \mathbb{I} \xi + u_c 
			\end{cases} \label{eqn:dynamics_X}
		\end{align}
		where $X\in SE(3)$ denotes the pose of a rigid body system, $\xi = [\omega\T,v\T]\T \in \mathbb{R}^6$  denotes the group velocity, $\mathbb{I} = \diag(J,mI_3)\in \mathbb{R}^{6\times 6}$ denotes the inertia matrix, and $u_c = [\tau\T,f\T]\T \in \mathbb{R}^6$ with $f$ and $\tau$ denoting the force and torque inputs, respectively.  Similar to \eqref{eqn:dynamics_R_r},  the desired reference trajectory is generated by the following dynamical system: 
		\begin{align}
			\left. \begin{array}{rl}
				\dot{X}_r  &= X_r \xi_r^\wedge\\
				\dot{\xi}_r & = z \\
				z  &\in m\mathbb{B}
			\end{array}\right\} (X_r,\xi_r)\in \mathcal{W}_d \label{eqn:dynamics_X_r}
		\end{align}
		where  $m\mathbb{B}:=\{x\in \mathbb{R}^6: \|x\|\leq m\}, m>0$, $\mathcal{W}_d$ denotes a compact subset of $SE(3)\times \mathbb{R}^6$,  and $X_r$ and $\xi_r$ are the desired pose and group velocity, respectively. 
		
		Define the pose tracking error $X_e= X_r^{-1} X$  and the group velocity tracking error $\xi_e = \xi - \Ad_{X_e}^{-1} \xi_r$. From \eqref{eqn:dynamics_X}-\eqref{eqn:dynamics_X_r}, one obtains the following error dynamics:
		\begin{subequations}\label{eqn:Xe-PIxi_e}
			\begin{align}
				\dot{X}_e & = X_e \xi_e^\wedge  \label{eqn:Xe}\\
				\mathbb{I}\dot{\xi}_e  &=  \Sigma(X_e,\xi_e,\xi_r)\xi_e  - \Upsilon(X_e, \xi_r,z) + u_c \label{eqn:PIxi_e} 
			\end{align}
		\end{subequations}
		where the functions $\Upsilon: SE(3)\times \mathbb{R}^6 \times \mathbb{R}^6 \to \mathbb{R}^6$ and  $\Sigma: SE(3)\times \mathbb{R}^6 \times \mathbb{R}^6 \to \mathbb{R}^{6\times 6}$ are given by
		\begin{subequations}
			\begin{align}
				\Sigma(X_e,\xi_e,\xi_r)&: =  \ad_{\xi_e }\T \mathbb{I}   - \ad_{\mathbb{I}\Ad_{X_e}^{-1}\xi_r}\T  \nonumber \\
				& ~~~~~~~  +  (\ad_{\Ad_{X_e}^{-1}\xi_r}\T \mathbb{I}   - \mathbb{I} \ad_{\Ad_{X_e}^{-1}\xi_r}  ) \\
				\Upsilon(X_e, \xi_r,z)&: = -\ad_{\Ad_{X_e}^{-1}\xi_r}\T \mathbb{I} \Ad_{X_e}^{-1}\xi_r+ \mathbb{I}\Ad_{X_e}^{-1}z. \label{eqn:UpsilonSE}
			\end{align}
		\end{subequations}
		Note that the error dynamics in \eqref{eqn:Xe-PIxi_e} have similar structure as in \eqref{eqn:Re-Jomega_e}, and the equality $\xi_e\T \Sigma(X_e,\xi_e,\xi_r)\xi_e = 0$ holds.  Let $V(X_e)$ be a potential function on $SE(3)$ with respect to $I_4$. Hence,  given the following smooth gradient-based feedback
		\begin{equation}
			u_c = \Upsilon(X_e, \xi_r,z) - 2 k_X\bar{\psi}(X_e^{-1} \nabla_{X_e}V(X_e)) - k_\xi \xi_e \label{eqn:smooth_u_c},
		\end{equation}
		the equilibrium point $(I_4,0)$ of the closed-loop system \eqref{eqn:Xe-PIxi_e}-\eqref{eqn:smooth_u_c} can be shown to be AGAS.

		Now, we will illustrate the difficulty of the application of the synergistic hybrid approach on $SE(3)$. Applying the  ``angular-warping" technique from \cite{mayhew2011synergistic}
		directly on  $SE(3)$, one has the transformation map $\mathcal{T}: SE(3) \to SE(3)$ as $\mathcal{T}(X_e)  = \exp(\vartheta(X_e) u^\wedge) X_e$ with $\vartheta(X_e) = k V(X_e)$ and $ k\in \mathbb{R},u\in \mathbb{R}^6$. Repeating the results in \cite[Theorem 6]{mayhew2011synergistic}, one obtains $\dot{\mathcal{T}}(X_e) = \mathcal{T}(X_e) (\bar{\Theta}(X_e)  \xi_e) ^\wedge$ with  $\bar{\Theta}(X_e)  = I_6 + 2k\Ad_{X_e}^{-1}   u    \bar{\psi}(X_e^{-1}\nabla_{X_e} V(X_e))\T$. 
		To guarantee that $\mathcal{T}$ is a diffeomorphism as in \cite[Theorem 8]{mayhew2011synergistic}, one way is to show that matrix $\bar{\Theta}(X_e)$ is invertible on $SE(3)$, \ie, $\det(\bar{\Theta}(X_e)) = 1 + 2|k|\|u\|\|\bar{\psi}(X_e^{-1}\nabla_{X_e} V(X_e))\| \neq 0$ for all $X_e\in SE(3)$. However, the choice of the scalar $k$ is difficult due to the fact that $SE(3)$ is non-compact and the upper bound of $\|\nabla_{X_e} V(X_e)\|_F$ cannot be \textit{a priori} determined.	
		To avoid this issue, an alternative design for hybrid feedback on $SE(3)$ with GAS guarantees has been proposed in \cite{casau2015globally}, which combines a hybrid attitude feedback relying on a synergistic family of potential functions on $SO(3)$ and a smooth linear feedback for the vector states. The key of this approach is that it separates the non-compact Lie group $SE(3)$ into a compact Lie group $SO(3)$ and a linear space $\mathbb{R}^3$ and, as such, the control is designed on $SO(3)\times \mathbb{R}^3$ rather than on $SE(3)$ directly. Our approach, however, is not restricted to compact manifolds and can handle the design of geometric hybrid control schemes directly on $SE(3)$.

		Let $U$ be a potential function on $SE(3)\times \mathbb{R}$ with respect to $\mathcal{A}_o':=(I_4,0)$ and $\Theta \subset \mathbb{R}$ be a nonempty finite set. We propose the following hybrid feedback tracking scheme: 
		\begin{align}
			\underbrace{
				\begin{array}{l}
					u_c = \Upsilon(X_e, \xi_r,z)    - k_\xi \xi_e\\
					\qquad   - 2 k_X\bar{\psi}(X_e^{-1} \nabla_{X_e}U(X_e,\theta)) \\
					\dot{\theta} ~= f(X_e,\theta)
				\end{array}
			}_{(X_e,\theta)\in \mathcal{F}}
			\underbrace{
				\begin{array}{l}	
					\\
					\\
					\\[-0.4cm] 	
					\theta^+  \in g(X_e, \theta)
				\end{array}
			}_{(X_e,\theta)\in \mathcal{J}}  \label{eqn:hybrid_feedback4}
		\end{align}
		where $k_X, k_\xi>0, \theta(0,0)\in \mathbb{R}$, $\Upsilon$ is defined in \eqref{eqn:UpsilonSE},  
		the flow map $f:SE(3)\times \mathbb{R} \to \mathbb{R}$ and the jump map $g:SE(3)\times \mathbb{R} \to \mathbb{R}$ are defined as
		\begin{align}
			f(X_e,\theta)&:= -  k_\theta\nabla_{\theta} U(X_e,\theta) \label{eqn:ftheta4} \\
			g(X_e,\theta) &:= \left\{\theta \in \Theta :  \theta= \arg \min\nolimits_{{\theta}'\in \Theta}U(X_e,{\theta}') \right\}  \label{eqn:gtheta4}
		\end{align}  	
		with $k_\theta>0$, and the flow and jump sets are given as
		\begin{subequations}
			\begin{align}
				\mathcal{F}&:= \left\{(X_e,\theta)\in SE(3) \times \mathbb{R}:  \mu_U(X_e,\theta)\leq \delta \right\} \label{eqn:Fset4} \\
				\mathcal{J}&:= \left\{(X_e,\theta) \in SE(3)\times \mathbb{R}:  \mu_U(X_e,\theta)\geq \delta \right\} \label{eqn:Jset4}
			\end{align} 
		\end{subequations}
		with some $\delta>0$ and the map $\mu_U: SE(3)\times \mathbb{R} \to \mathbb{R}$ given as
		\begin{equation}
			\mu_U(X_e,\theta)  := U(X_e,\theta) - \min\nolimits_{{\theta}'\in \Theta} U(X_e,\theta') \label{eqn:mu_USE}.
		\end{equation}
		The proposed hybrid feedback \eqref{eqn:hybrid_feedback4} is modified from \eqref{eqn:hybrid_feedback1} and designed on $SE(3)$, in terms of the geometric tracking errors $X_e,\xi_e$ and a general potential function $U$ on $SE(3)\times \mathbb{R}$. 
		Now, one can state the following result:
		\begin{thm}\label{them:4}
			Let $k_X, k_\xi,k_\theta>0$ and suppose that  there exist a potential function $U$ on $SE(3)\times \mathbb{R}$ with respect to $\mathcal{A}_o'$  and a nonempty finite set  $\Theta \subset \mathbb{R}$ such that  $\mathcal{A}_o' \in \Psi_U$ with $\Psi_U$ denoting the set of all critical points of $U$ and
			\begin{equation}
				\mu_U(X,\theta)  > \delta,  \quad  \forall (X,\theta)\in \Psi_U\setminus\{\mathcal{A}_o'\}  \label{eqn:u_U_delta4}
			\end{equation} 
			with some constant $\delta >0$ and $\mu_U$ defined in \eqref{eqn:mu_USE}. 
			Then,  the set $\mathcal{A}':= \{x \in SE(3) \times \mathbb{R} \times \mathbb{R}^6 \times \mathcal{W}_d  : (X_e,\theta)= \mathcal{A}_o', \xi_e=0 \} $ is globally asymptotically stable for the closed-loop   system  \eqref{eqn:Xe-PIxi_e} with the hybrid feedback \eqref{eqn:hybrid_feedback4}, and the number of jumps is finite.  
		\end{thm}
		The proof of Theorem \ref{them:4} can be conducted by following the same steps as in the proof of Theorem \ref{them:1}, and therefore is omitted here. It is important to point out that the key condition of Theorem \ref{them:4} is that the basic Assumption \ref{assum:1} holds for $SE(3)$ (\ie, inequality \eqref{eqn:u_U_delta4}). To complete the hybrid feedback design on $SE(3)$, we need to construct a potential function $U$ on $SE(3)\times \mathbb{R}$ such that inequality \eqref{eqn:u_U_delta4} holds.

		Consider the following transformation map  $\mathcal{T}: SE(3)\times \mathbb{R} \to SE(3)$ :
		\begin{equation}
			\mathcal{T}(X,\theta) :=  X \exp(\theta \bar{u}^\wedge)   \label{eqn:transfT2}
		\end{equation}
		where $X\in SE(3)$, $\bar{u}\in \mathbb{R}^6$ is a constant vector and  $\theta \in \mathbb{R}$ is a real-valued variable with hybrid dynamics specified in \eqref{eqn:hybrid_feedback4}-\eqref{eqn:gtheta4}. For the sake of simplicity, let 
		$\mathcal{T}_{X,\theta}:=\mathcal{T}(X,\theta)$. 
		Let us introduce the following potential function on $SE(3)\times \mathbb{R}$ with respect to $\mathcal{A}_o'$ 
		\begin{align}
			U(X,\theta) 
			&= \frac{1}{2}\tr((I_4-\mathcal{T}_{X,\theta})\mathbb{A}(I_4-\mathcal{T}_{X,\theta})\T) + \frac{\gamma }{2  } \theta^2    \label{eqn:UXtheta}	
		\end{align}
		with a symmetric positive definite matrix $\mathbb{A}\in \mathbb{R}^{4\times 4}$ and a constant $\gamma >0$ to be designed. Let $\Psi_V $ denote the set of critical points of potential function $V(X) = \frac{1}{2}\tr((I_4-X)\mathbb{A}(I_4-X)\T)$ on $SE(3)$, which can be computed as per \cite[Lemma 5]{wang2019hybrid}. 
		The following proposition provides some useful properties of the potential function $U$ on $SE(3)\times \mathbb{R}$:	
		\begin{prop}
			Let   $\mathbb{A} = \mathbb{A}\T$ be a positive definite matrix.   Consider  the potential function $U$ defined in \eqref{eqn:UXtheta}, and the trajectories generated by $\dot{X}=X\xi^\wedge$ and $\dot{\theta}=\nu$ with $X(0)\in SE(3), \theta(0)\in \mathbb{R}, \xi \in \mathbb{R}^6, \nu\in \mathbb{R}$. Then, for all $(X,\theta)\in SE(3)\times \mathbb{R}$, the following statements hold:
			\begin{subequations}
				\begin{align}   
					\dot{\mathcal{T}}_{X,\theta}
					& = \mathcal{T}_{X,\theta}  (\Ad_{\exp(\theta \bar{u}^\wedge)}^{-1} \xi + \nu \bar{u})^\wedge  \label{eqn:dTXtheta} \\
					\bar{\psi}(X^{-1} \nabla_{X} U(X,\theta)) &=   \Ad_{\exp(\theta \bar{u}^\wedge)}^{-\top} \bar{\psi}( (I_4-\mathcal{T}_{X,\theta}^{-1})\mathbb{A})  \label{eqn:gradient_X_U}  \\
					\nabla_{\theta} U(X,\theta) & = \gamma \theta + 2\bar{u}\T \bar{\psi}( (I_4-\mathcal{T}_{X,\theta}^{-1})\mathbb{A}) \label{eqn:gradient_theta_UX} \\
					\mathcal{A}_o' & \in \Psi_{U}  := \Psi_{V} \times \{0\} 
				\end{align}
			\end{subequations}
			Moreover, choosing $\bar{u} = [u\T,0]\T$, $\mathbb{A} = \diag(A,1)$ and the set of parameters $\mathcal{P}_U=\{\Theta,A,u,\gamma,\delta\}$ defined in \eqref{eqn:PU} one has
			$$
			U(X,\theta) = \tr((I_3-R\mathcal{R}_{a}(\theta,u))A) + \frac{1}{2}\|p\|^2 + \frac{\gamma }{2  } \theta^2 
			$$
			and the condition \eqref{eqn:u_U_delta4} in Theorem \ref{them:4} holds. 
		\end{prop}
		The proof of this Proposition can be conducted using similar steps as in the proof of Lemma \ref{lemma:U} and Proposition \ref{prop:1}, and hence omitted.

	\section{Simulation} \label{sec:simulation} 
	In this section, numerical simulations are presented to illustrate the performance of the proposed hybrid feedback controllers. We make use of the HyEQ Toolbox in Matlab \cite{sanfelice2013toolbox}. The hybrid controller \eqref{eqn:hybrid_feedback1} is referred to as `Basic Hybrid', the hybrid controller with torque smoothing mechanism in \eqref{eqn:hybrid_feedback2} is referred to as  `Smooth Hybrid', and the velocity-free hybrid controller in \eqref{eqn:hybrid_feedback3} is referred to as `Velocity-Free Hybrid'.  For comparison purposes, we also consider the following classical smooth non-hybrid controller, referred to as `Non-Hybrid':
	\begin{align}
		\tau = \Upsilon(R_e, \omega_r,z)   -2 k_R\psi(A R_e) - k_\omega  \omega_e \label{eqn:smooth_feedback}
	\end{align}
	which is modified from the hybrid controller \eqref{eqn:hybrid_feedback1} by taking $\theta \equiv 0$.  The inertia matrix of the  system is taken as  $J = \diag(0.0159,0.0150,0.0297)$  obtained from a quadrotor UAV in \cite{wang2015attitude}. The reference rotation and angular velocity are generated by \eqref{eqn:dynamics_R_r} with $R_r(0)=I_3$, $\omega_r(0)=0$ and $z(t)=[\sin(0.1t),-\cos(0.3t),0.1]\T$. For the set $\mathcal{P}_U$ in Proposition \ref{prop:3}, we choose $\Theta = \{\theta_M\}$ with $\theta_M= 0.9\pi$, $A = \diag([2, 4, 6])$, $\gamma <   \frac{4\Delta^*}{\pi^2} =   \frac{8}{\pi^2} $, $\delta < ( \frac{4\Delta^*}{\pi^2} -\gamma) \frac{{\theta_M}^2}{2}$,  $u = [0,\sqrt{{2}/{5}}, \sqrt{ {3}/{5}}]\T$ and $\Delta^* = \lambda_1^A=2$ as per case 2) in Proposition \ref{prop:3} (\ie, $\lambda_2^A >  {\lambda_1^A \lambda_3^A}/(\lambda^A_3 - \lambda_1^A)$). 

	In our first simulation, three different choices of $\gamma$ such as $\frac{3}{\pi^2},\frac{5}{\pi^2},\frac{7}{\pi^2}$ are considered in the hybrid controller \eqref{eqn:hybrid_feedback1}. For each $\gamma$, the constant gap $\delta$ is chosen as $\delta =  \frac{4}{10}( \frac{4\Delta^*}{\pi^2} -\gamma) \theta_M^2$. Moreover, the gain parameters are chosen as $k_R = 1.5, k_\omega = 0.2, k_\theta = 50$, and the initial conditions are  chosen as $\omega(0) = 0$, $R(0) = \mathcal{R}_a(\pi-\epsilon,u)$ with $\epsilon=10^{-9},u=[0,0,1]\T$ and $\theta(0,0)=0$, which ensures that initial $(R_e,\theta)$ is close to one of the undesired critical points of $U$. Same gains and initial conditions are considered in the non-hybrid controller \eqref{eqn:smooth_feedback}. The simulation results are given in Fig. \ref{fig:simulation1}.  As one can see, for the basic hybrid feedback, the variable $\theta$ in \eqref{eqn:dynamics_theta}  jumps from 0 to $0.9\pi$ at $t=0$ and then converges to zero as $t\to \infty$. Moreover, the tracking errors $(R_e,\omega_e)$ of both controllers converge to zero as $t\to \infty$. One can also see that the hybrid controller \eqref{eqn:hybrid_feedback1} improves the convergence rate as compared to the  non-hybrid controller \eqref{eqn:smooth_feedback}, and an increase in the value of $\gamma$ leads to an increase in the convergence rate of the tracking errors.
	
	\begin{figure}[!ht] 
		\centering 
		\includegraphics[width=0.9\linewidth]{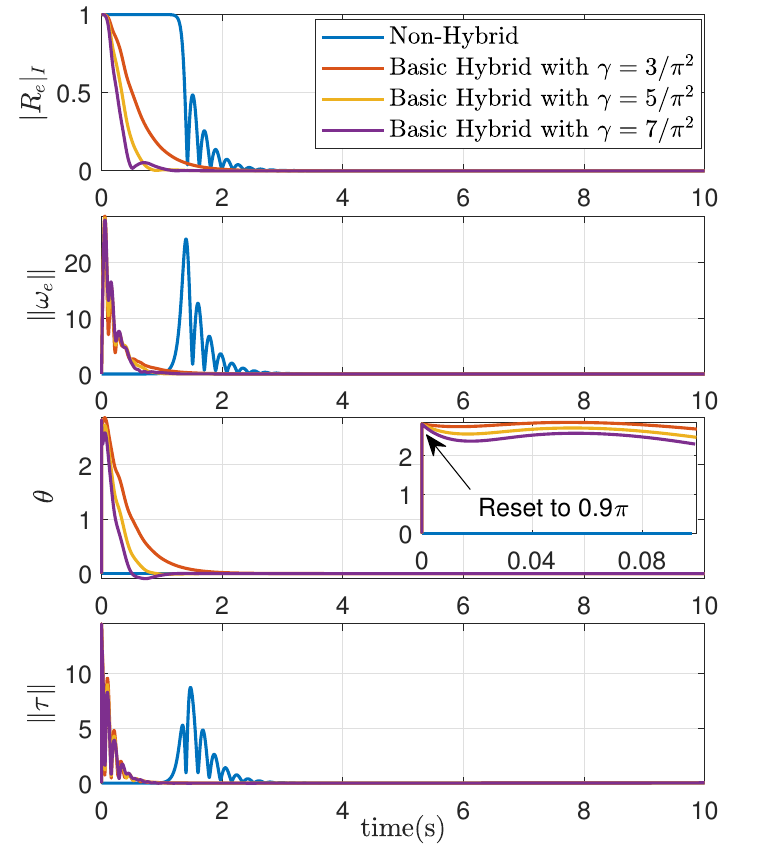}   
		\caption{Simulation results for the  hybrid feedback \eqref{eqn:hybrid_feedback1} with   different choices of parameter $\gamma$.
		}
		\label{fig:simulation1}
	\end{figure}

	Our second simulation presents a comparison between the three proposed hybrid controllers in the presence of measurements noise. The noisy measurements of attitude and angular velocity are given as $R_y = R\exp(n_R^\times)$ with zero-mean Gaussian noise $n_R \sim \mathcal{N}(0,0.01 I_3)$, and $\omega_y = \omega + n_\omega$ with zero-mean Gaussian noise $n_\omega \sim \mathcal{N}(0,0.01 I_3)$. Same initial conditions for $R,\theta$ and $\omega$ are chosen as in the previous simulation, and in addition $\bar{R}(0) =  R(0)\T$, $ \bar{\theta}(0,0)=0$ and $\zeta(0)=0$ are considered. We choose $\gamma = \frac{7}{\pi^2}$ and $ \delta =   \frac{4}{10}( \frac{4\Delta^*}{\pi^2} -\gamma) \theta_M^2 = 0.324$ for  $U$,   and $\delta'=0.162$ and $\varrho = 0.0146$ for  $W$ defined in \eqref{eqn:defW}.  Moreover,  the gain parameters  are chosen as follows:
	\begin{tabbing} 
		\centering  
		\begin{tabular}{|c|c|c|c|c|c|c|}
			\hline
			& $k_R$ & $k_\omega$ & $k_\theta$ & $k_\zeta$ & $k_\beta$ & $\Gamma$ \\
			\hline
			Basic Hybrid & 1.5 & 0.2 & 50 & $-$ & $-$ &$-$ \\
			\hline
			Smooth Hybrid & 1.5 & 0.2 & 50 & 150 & $-$ & $-$ \\
			\hline
			Velocity-Free Hybrid & 1.5 & $-$  & 50 & $-$ & 3 & 30$I_3$  \\
			\hline
		\end{tabular}
	\end{tabbing}
	Note that 
	$k_\beta$ and $\Gamma$ are chosen such that $2k_\beta \Gamma^{-1} = k_\omega$. 
	The simulation results are given in Fig. \ref{fig:simulation2}. For all controllers, the tracking errors $R_e,\omega_e$ and $\theta$ converge to zero, after one second. 
	Through an appropriate gain tuning, the three hybrid tracking controllers exhibit a quite similar performance. Note that the velocity-free hybrid controller is more sensitive to noise as shown in the plot of the control torque, which is mainly due to the large gain $\Gamma$ involved in the an auxiliary system \eqref{eqn:bar_R} to overcome the lack of angular velocity measurements.

	\begin{figure}[!ht] 
		\centering 
		\includegraphics[width=0.9\linewidth]{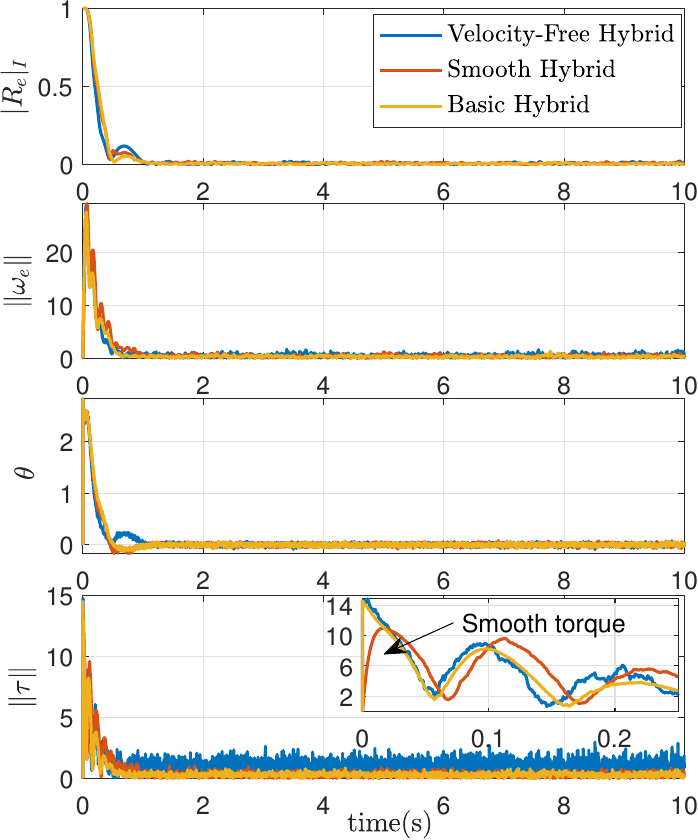}  
		\caption{Simulation results for different  hybrid feedback with noisy measurements.
		}
		\label{fig:simulation2}
	\end{figure}	
	
	\section{Conclusion} \label{sec:conclusion}
	Three different hybrid feedback control schemes for the attitude tracking problem on $SO(3)$, leading to global asymptotic stability, have been proposed. As an instrumental tool in our design, a new potential function on $SO(3)\times \mathbb{R}$, involving a potential function on $SO(3)$ and a scalar variable $\theta$, has been proposed. The scalar variable  $\theta$ is governed by hybrid dynamics designed to prevent the extended state in $SO(3)\times \mathbb{R}$ from reaching the undesired critical points, while guaranteeing a decrease of the potential function after each jump. In fact, embedding the manifold $SO(3)$  in the higher dimensional space $SO(3)\times \mathbb{R}$ allows to modify the critical points on $SO(3)$ by tying them to $\theta=0$. This embedding mechanism  provides an easier handling of the critical points on the extended manifold $SO(3)\times \mathbb{R}$ through the hybrid dynamics of the scalar variable $\theta$.\\
	A global hybrid attitude tracking controller is designed from the gradient of the potential function with the full knowledge of the system state. For practical purposes, two extensions have been proposed: A hybrid attitude tracking controller with jump-free control torque  and a velocity-free hybrid attitude tracking controller.  The proposed hybrid strategy, involving a single potential function on $SO(3)\times \mathbb{R}$, on top of being simpler than the existing hybrid approaches involving a synergistic family of potential functions, shows a great potential for other applications involving non-compact manifolds where the synergistic hybrid approaches may not be applicable.  This fact has been demonstrated through the design of a globally asymptotically stabilizing (geometric) hybrid feedback for the tracking control problem on the non-compact manifold $SE(3)$.

	\appendix

	\subsection{Proof of Theorem \ref{them:1}} \label{sec:them1}
	Consider the following Lyapunov function candidate:
	\begin{equation}
		\mathcal{L}(x) =  k_R U(R_e,\theta) + \frac{1}{2} \omega_e\T J \omega_e. \label{eqn:Lx}
	\end{equation}
	Since $U$ is a potential function on $SO(3)\times \mathbb{R}$ with respect to $\mathcal{A}_o$, and $J=J\T$ is positive definite,  one can verify that $\mathcal{L}$ is positive definite on $\mathcal{S}$ with respect to $\mathcal{A}$. The time derivative of $\mathcal{L}$ along the flows of \eqref{eqn:closed-loop1} is given by
	\begin{align}
		\dot{\mathcal{L}}(x)
		&=   k_R\dot{U}(R_e,\theta)- \omega_e\T  \kappa(R_e,\theta,\omega_e) \label{eqn:dotL}
	\end{align}
	where we used $\omega_e\T  \Sigma(R_e,\omega_e,\omega_r)\omega_e =0$. From \eqref{eqn:Re} and \eqref{eqn:dynamics_theta}, one obtains 
	\begin{align}
		\dot{U}(R_e,\theta)
		& =  \langle \nabla_{R_e} U(R_e,\theta), R_e \omega_e^\times \rangle_{R_e}  + \langle\langle \nabla_{\theta} U(R_e,\theta), \dot{\theta}\rangle\rangle \nonumber \\
		& = \langle\langle R_e\T \nabla_{R_e} U(R_e,\theta), \omega_e^\times \rangle\rangle  +    \nabla_{\theta} U(R_e,\theta) \dot{\theta}   \nonumber \\
		& = 2 \omega_e\T \psi (R_e\T \nabla_{R_e} U(R_e,\theta))   -   k_\theta |\nabla_{\theta} U(R_e,\theta)|^2  \label{eqn:dotU_Re}
	\end{align}
	where we made use of the property $\langle\langle A, x^\times\rangle\rangle = 2x\T \psi (A) $. 
	Substituting  \eqref{eqn:kappa} and \eqref{eqn:dotU_Re} into \eqref{eqn:dotL}, one can further show that
	\begin{align}
		\dot{\mathcal{L}}(x)
		& = 2 k_R \omega_e\T \psi (R_e\T \nabla_{R_e} U(R_e,\theta)) - k_R k_\theta |\nabla_{\theta} U(R_e,\theta)|^2 \nonumber\\
		& ~~~ - \omega_e\T (2 k_R \psi(R_e\T \nabla_{R_e} U(R_e,\theta)) + k_\omega  \omega_e) \nonumber\\
		&= -k_\omega \|\omega_e\|^2   -  k_R k_\theta |\nabla_{\theta} U(R_e,\theta) |^2  \leq 0 \label{eqn:dotLx}
	\end{align}
	for all $x\in \mathcal{F}_c$. 
	Thus, $\mathcal{L}$ is non-increasing along the flows of \eqref{eqn:closed-loop1}. Moreover, in view of   \eqref{eqn:closed-loop1} and \eqref{eqn:Lx}, for any $x \in \mathcal{J}_c$, one has $x^+ \in (R_e, g(R_e,\theta),\omega_e,R_r,\omega_r) $ and
	\begin{align}
		\mathcal{L}(x^+) - \mathcal{L}(x) 
		&= -k_R \left( U(R_e,\theta) -   \min_{{\theta}'\in \Theta} U(R_e,{\theta}') \right)  \nonumber \\
		& = -k_R \mu_U(R_e,\theta)  < -k_R\delta \label{eqn:Lx+}
	\end{align}
	where we made use of the fact  $\mu_U(R_e,\theta) > \delta$ for all $x \in  \mathcal{J}_c$.
	Thus, $\mathcal{L}$ is strictly decreasing over the jumps of \eqref{eqn:closed-loop1}. From \eqref{eqn:dotLx} and \eqref{eqn:Lx+}, one concludes that the set ${\mathcal{A}}$ is stable as per \cite[Theorem 23]{goebel2009hybrid}, and every maximal solution to \eqref{eqn:closed-loop1} is bounded. Moreover, in view of \eqref{eqn:dotLx} and \eqref{eqn:Lx+}, one obtains
	$
	\mathcal{L}(x(t,j))   \leq \mathcal{L}(x(t_j,j)) $ and $
	\mathcal{L}(x(t_j,j))  \leq \mathcal{L}(x(t_j,j-1))-k_R\delta
	$
	for all $(t,j),(t_j,j),(t_j,j-1)\in \dom x $ with $ (t,j) \succeq (t_j,j) \succeq (t_j,j-1)$. Hence, it is clear that
	$ 0\leq \mathcal{L}(x(t,j))   \leq \mathcal{L}(x(0,0)) - j k_R \delta $ for all $ (t,j)\in \dom x $, 
	which leads to
	$
	j\leq  \left\lceil {\mathcal{L}(x(0,0))}/{(k_R \delta)}\right\rceil:=J_M,
	$
	where $\lceil\cdot\rceil$ denotes the ceiling function. This shows that the number of jumps is finite and depends on the initial conditions.

	Next, we will show the global attractivity of $\mathcal{A}$.
	Applying the invariance principle for hybrid systems given in \cite[Theorem 4.7]{sanfelice2007invariance}, one concludes from  \eqref{eqn:dotLx} and \eqref{eqn:Lx+}  that any solution $x$ to the hybrid system \eqref{eqn:closed-loop1} must converge to the largest invariant set contained in
	$
	\mathcal{W}:
	=\left\{x\in \mathcal{F}_c ~|~ \nabla_{\theta} U(R_e,\theta)=0,  \omega_e=0 \right\}.
	$
	For each $x\in \mathcal{W}$, from $\omega_e\equiv 0$ one has $\dot{\omega}_e = 0$. It follows from  \eqref{eqn:Jomega_e}, \eqref{eqn:hybrid_feedback1} and \eqref{eqn:kappa}  that $\psi(R_e\T \nabla_{R_e} U(R_e,\theta)) = 0$. Using this fact, together with $ \nabla_{\theta} U(R_e,\theta) \equiv 0$, one can show that $(R_e,\theta)\in \Psi_U$ with $\Psi_U$ defined in \eqref{eqn:criticalU}. Thus, any solution $x$ to the hybrid system \eqref{eqn:closed-loop1} must converge to the largest invariant set contained in $ {\mathcal{W}}'
	:=\left\{x\in \mathcal{F}_c ~|~  (R_e,\theta)\in \mathcal{F} \cap \Psi_U, \omega_e=0\right\}.$ 
	By Assumption \ref{assum:1},  one has $\mathcal{A}_o \in \Psi_U$ and $\mu_U(R_e,\theta) = - \min_{{\theta}'\in \Theta} U(R_e,{\theta}')\leq 0$ as $(R_e, {\theta}) = \mathcal{A}_o$. It follows from \eqref{eqn:u_U_delta} and \eqref{eqn:Fset}-\eqref{eqn:Jset} that $\mathcal{A}_o \in \mathcal{F} \cap \Psi_U$ and $\mathcal{F} \cap (\Psi_U\setminus\{\mathcal{A}_o\}) =\emptyset$.
	Then, applying simple set-theoretic arguments, one obtains
	$
	\mathcal{F} \cap \Psi_U  \subset ( \mathcal{F} \cap (\Psi_U\setminus\{\mathcal{A}_o\}) ) \cup (\mathcal{F} \cap \{\mathcal{A}_o\}   )  =   \emptyset \cup \{\mathcal{A}_o\} = \{\mathcal{A}_o\}.
	$
	It follows from $\mathcal{A}_o \in \mathcal{F} \cap \Psi_U$ and $\mathcal{F} \cap \Psi_U  \subset  \{\mathcal{A}_o\}$ that $\mathcal{F} \cap \Psi_U = \{\mathcal{A}_o \}$. Consequently, from the definitions of ${\mathcal{W}}'$ and $\mathcal{A}$, it follows that ${\mathcal{W}}' = \mathcal{A}$. 
	
	Note that the closed-loop system \eqref{eqn:closed-loop1} satisfies the hybrid basic conditions \cite[Assumption 6.5]{goebel2012hybrid}, $F(x) \subset T_{\mathcal{F}}(x)$ for any $x\in \mathcal{F}_c \setminus \mathcal{J}_c$ with $T_{\mathcal{F}}(x)$ denoting the tangent cone to $\mathcal{F}_c$ at the point $x$, $G(\mathcal{J}_c) \subset \mathcal{F}_c \cup \mathcal{J}_c = \mathcal{S}$, and every maximal solution to \eqref{eqn:closed-loop1} is bounded.  Therefore, by virtue of \cite[Proposition 6.10]{goebel2012hybrid}, it follows that every maximal solution to \eqref{eqn:closed-loop1} is complete. Finally, one can conclude that the set $\mathcal{A}$ is globally asymptotically stable for the hybrid system \eqref{eqn:closed-loop1}.  This completes the proof.
	
	\subsection{Proof of Proposition \ref{prop:1}} \label{sec:prop1}
	Consider  the following Lyapunov function candidate:
	\begin{align}
		\mathcal{L}_{\varepsilon} (x)  =  \mathcal{L}(x) + \varepsilon \omega_e\T J \psi(R_e\T \nabla_{R_e} U(R_e,\theta))\label{eqn:L_varepsilonx} 
	\end{align}	
	where $\varepsilon>0$ and $\mathcal{L}(x)$ is given in \eqref{eqn:Lx}. From \eqref{eqn:assump2_2}, one has $\|\psi(R_e\T \nabla_{R_e} U(R_e,\theta))\|^2\leq   \alpha_1 U(R_e,\theta)$, and consequently, one can show that 
	\begin{align}
		\lambda_{\min}^{P_1} |x|_{\mathcal{A}}^2 \leq \mathcal{L}_{\varepsilon}(x) \leq \lambda_{\max}^{P_2} |x|_{\mathcal{A}}^2, \quad  \forall x\in \mathcal{S} \label{eqn:L_varepsilonx_bound} 
	\end{align}
	where matrices $P_1$ and $P_2$ are given as
	\begin{align*}    \setlength\arraycolsep{0.0pt}
		P_1 =\begin{bmatrix}   
			k_R   & \frac{-\varepsilon \sqrt{\alpha_1} \lambda_{\max}^J }{2}  \\
			\frac{-\varepsilon \sqrt{\alpha_1} \lambda_{\max}^J }{2}& \frac{\lambda_{\min}^J}{2}  
		\end{bmatrix},  
		P_2 =\begin{bmatrix}   
			k_R  &  \frac{\varepsilon \sqrt{\alpha_1} \lambda_{\max}^J }{2}  \\
			\frac{\varepsilon \sqrt{\alpha_1} \lambda_{\max}^J }{2}  & \frac{\lambda_{\max}^J}{2} 
		\end{bmatrix}.
	\end{align*} 
	To guarantee that $P_1$ and $P_2$ are positive definite, it is sufficient to choose $\varepsilon< ({1}/{\lambda_{\max}^J}) \sqrt{2k_R \lambda_{\min}^J /\alpha_1}:=\varepsilon^*_1$. 
	
	Since the set $ \varOmega_c  \times \mathcal{W}_d$ is compact by assumption, there exists a constant scalar $\mathcal{L}^*\geq 0$ such that $\mathcal{L}^*:=\sup_{x\in \varOmega_c  \times \mathcal{W}_d} \mathcal{L}(x)$. We define the following compact set
	$
		\varOmega_{\mathcal{L}}: = \{x\in \mathcal{S}: \mathcal{L}(x) \leq \mathcal{L}^*\}.
	$ 
	It is clear that $x(0,0)\in  \varOmega_c  \times \mathcal{W}_d \subseteq \varOmega_{\mathcal{L}}$ and $\mathcal{A} \subseteq \varOmega_{\mathcal{L}}$. As shown in the proof of Theorem \ref{them:1}, $\mathcal{L}(x)$ is non-increasing in both flow and jump sets. Hence, for any $x(0,0)\in \varOmega_c  \times \mathcal{W}_d \subset \mathcal{S}$, one has $x(t,j)\in \varOmega_{\mathcal{L}}$ for all $(t,j) \in \dom x$ and the number of jumps is bounded by $J_M:=\left\lceil {\mathcal{L}^*}/{(k_R \delta)}\right\rceil$. Using the facts $ \|\psi(R_e\T \nabla_{R_e} U(R_e,\theta))\|^2 \leq  \alpha_1 U(R_e,\theta) \leq \frac{\alpha_1}{ k_R}\mathcal{L}^*$ and $   \omega_e\T J \omega_e \leq 2\mathcal{L}^*$, it follows that there exist constants $c_{\psi},c_{\omega_e}>0$  such that $\|\psi(R_e\T \nabla_{R_e} U(R_e,\theta))\| \leq c_{\psi}$ and $\|\omega_e\| \leq c_{\omega_e}$ for all $(t,j)\in \dom x$.  Let $c_{\omega_r}: = \sup_{t\geq 0} \|\omega_r(t)\|$ since  $\mathcal{W}_d$ is compact by assumption. 
	Hence, from \eqref{eqn:Sigma_def}, \eqref{eqn:kappa}-\eqref{eqn:closed-loop1-F} and \eqref{eqn:dpsi}, for all  $x\in \varOmega_{\mathcal{L}} \cap \mathcal{F}_c$ one obtains  
	\begin{align} 
		&\frac{d}{dt} \omega_e\T J \psi(R_e\T \nabla_{R_e} U(R_e,\theta))  \nonumber \\
		&~ \leq    (\Sigma(R_e,\omega_e,\omega_r)\omega_e -\kappa(R_e,\theta,\omega_e) )\T \psi(R_e\T \nabla_{R_e} U(R_e,\theta))  \nonumber \\ 
		& \qquad +  \lambda_{\max}^J \|\omega_e\|   ( c_R \|\omega_e\| + c_\theta k_\theta |\nabla_{\theta} U(R_e,\theta)| ) \nonumber
		\\
		&~ \leq    c_{\psi} \lambda_{\max}^J  \|\omega_e\|^2 + 3   c_{\omega_r} \lambda_{\max}^J \|\omega_e\|   \|\psi(R_e\T \nabla_{R_e} U(R_e,\theta))\| \nonumber \\
		&\qquad - 2 k_R\|\psi(R_e\T \nabla_{R_e} U(R_e,\theta))\|^2 \nonumber \\ 
		& \qquad  + k_\omega  \|\omega_e\|\|\psi(R_e\T \nabla_{R_e} U(R_e,\theta))\|  \nonumber \\ 
		& \qquad +   \lambda_{\max}^J\|\omega_e\|   ( c_R \|\omega_e\| + c_\theta k_\theta |\nabla_{\theta} U(R_e,\theta)| ) \nonumber
		\\
		&~ \leq  \eta\T \setlength\arraycolsep{2.6pt}
		\underbrace{\begin{bmatrix}  
				-2k_R & 0 & \frac{(3  c_{\omega_r} \lambda_{\max}^J+ k_\omega)}{2}\\
				0 & 0 &  \frac{  c_\theta k_\theta \lambda_{\max}^J}{2} \\
				\frac{(3 c_{\omega_r} \lambda_{\max}^J  + k_\omega)}{2} &  \frac{c_\theta k_\theta \lambda_{\max}^J  }{2}   & \lambda_{\max}^J(c_{\psi} + c_R)
		\end{bmatrix}}_{P_{\varepsilon}} \eta
		\label{eqn:domega_psi} 
	\end{align}
	where $\eta := [\|\psi(R_e\T \nabla_{R_e} U(R_e,\theta))\|,|\nabla_{\theta} U(R_e,\theta)|,\|\omega_e\|]\T \in \mathbb{R}^3$, and the following facts $\|\psi(R_e\T \nabla_{R_e} U(R_e,\theta)) \|  \leq c_\psi$, $\|\omega_r\|\leq c_{\omega_r}$, and  
	$
	(\Sigma(R_e,\omega_e,\omega_r)\omega_e  )\T \psi(R_e\T \nabla_{R_e} U(R_e,\theta))  
	=   ((J\omega_e)^\times \omega_e   +       (JR_e\T \omega_r)^\times \omega_e -  (R_e\T \omega_r)^\times J \omega_e  + J(R_e\T \omega_r)^\times \omega_e) \T  \\  
	\psi(R_e\T \nabla_{R_e} U(R_e,\theta)) \leq    3 c_{\omega_r} \lambda_{\max}^J  \|\omega_e\| \|\psi(R_e\T \nabla_{R_e} U(R_e,\theta))\| \\ +   c_\psi \lambda_{\max}^J \|\omega_e\|^2$, 	
	were used.
	
	Therefore, from \eqref{eqn:dotLx} and  \eqref{eqn:domega_psi},  the time derivative of $\mathcal{L}_{\varepsilon}$ along the flows of \eqref{eqn:closed-loop1} can be written as
	
	\begin{align}		
		\dot{\mathcal{L}}_{\varepsilon}(x)
		&= \dot{\mathcal{L}}(x) + \varepsilon \frac{d}{dt} \omega_e\T J \psi(R_e\T \nabla_{R_e} U(R_e,\theta))  \nonumber \\ 
		&  \leq -k_\omega \|\omega_e\|^2   -  k_R k_\theta \|\nabla_{\theta} U(R_e,\theta)   \|^2  - \varepsilon \eta\T P_{\varepsilon} \eta  \nonumber \\
		&  = -\eta\T \underbrace{\left(  \begin{bmatrix}
				0 & 0 & 0\\
				0 &  k_R k_\theta & 0\\
				0 & 0 & k_\omega
			\end{bmatrix} - \varepsilon P_{\varepsilon} \right)}_{P_3} \eta 
		\label{eqn:dL_varepsilonx_1}
	\end{align}	 
	for all $x\in \varOmega_{\mathcal{L}} \cap \mathcal{F}_c$.
	Let   $\eta_{ij}=[\eta_i,\eta_j]\T, i,j \in \{1,2,3\}$ with $\eta_i$ denoting the $i$-th element of the vector $\eta$. From \eqref{eqn:dL_varepsilonx_1} and the definition of $P_{\varepsilon}$ in \eqref{eqn:domega_psi}, one has
	\begin{align*}
		\eta\T P_3 \eta & = -\eta_{13}\T  \begin{bmatrix}
			2\varepsilon k_R & -\frac{\varepsilon(3   c_{\omega_r} \lambda_{\max}^J + k_\omega)}{2} \\
			-\frac{\varepsilon(3   c_{\omega_r} \lambda_{\max}^J + k_\omega)}{2} & \frac{ k_\omega - 2\varepsilon \lambda_{\max}^J  (c_{\psi} + c_R) }{2}
		\end{bmatrix} \eta_{13}  \nonumber \\
		&\quad   -\eta_{23}\T   \begin{bmatrix}
			k_R k_\theta & -\frac{\varepsilon   c_\theta k_\theta \lambda_{\max}^J }{2} \\
			-\frac{\varepsilon   c_\theta k_\theta \lambda_{\max}^J }{2} & \frac{k_\omega}{2}
		\end{bmatrix} \eta_{23}  
	\end{align*}
	To ensure that  $P_3$ is positive definite,   it is sufficient to choose  
	$ 
	\varepsilon  <  
	\min   \{\frac{4k_R k_\omega}{  (3   c_{\omega_r} \lambda_{\max}^J + k_\omega)^2 + 8k_R \lambda_{\max}^J(c_{\psi} + c_R)},  \frac{2\sqrt{k_\omega k_R k_\theta}}{  c_\theta k_\theta\lambda_{\max}^J}  \}:=\varepsilon^*_2.
	$
	From  \eqref{eqn:assump2_3}, one can show that for any $x\in \mathcal{F}_c$, $\|\eta\|^2 = \|\psi(R_e\T \nabla_{R_e} U(R_e,\theta))\|^2 + |\nabla_{\theta} U(R_e,\theta)|^2 + \|\omega_e\|^2 \geq \alpha_2 U(R_e,\theta) + \|\omega_e\|^2 \geq c_\eta |x|_{\mathcal{A}}^2$ with $c_\eta: = \min\{\alpha_2,1\}$. 
	Hence, from the definition of $|x|_{\mathcal{A}}^2$ and \eqref{eqn:dL_varepsilonx_1}, the time derivative of $\mathcal{L}_{\varepsilon}$ can be rewritten as
	\begin{align}
		\dot{\mathcal{L}}_{\varepsilon}(x) & \leq -\lambda_{\min}^{P_3} \|\eta\|^2  \leq -  c_\eta  \lambda_{\min}^{P_3}|x|_{\mathcal{A}}^2  , ~ \forall x\in \varOmega_{\mathcal{L}} \cap \mathcal{F}_c		.  \label{eqn:dL_varepsilonx}
	\end{align} 
	Thus, $\mathcal{L}_{\varepsilon}$  has an exponential decrease over the flows of \eqref{eqn:closed-loop1}.
	
	On the other hand, from \eqref{eqn:closed-loop1-J}, \eqref{eqn:Lx+} and \eqref{eqn:L_varepsilonx} one obtains
	\begin{align}
		\mathcal{L}_{\varepsilon}(x^+) - \mathcal{L}_{\varepsilon}(x)  & \leq 
		\mathcal{L}(x^+)  - \mathcal{L}(x) + 2 \varepsilon   c_\psi c_{\omega_e} \lambda_{\max}^J \nonumber \\	 
		& \leq -k_R \delta  + 2 \varepsilon   c_\psi c_{\omega_e}  \lambda_{\max}^J \nonumber \\
		& <  0  , \quad \forall x\in \varOmega_{\mathcal{L}} \cap  \mathcal{J}_c \label{eqn:L_varepsilonx+} 
	\end{align}
	where $\varepsilon$ is chosen as $\varepsilon<\min\{\varepsilon^*_1,\varepsilon^*_2, \frac{k_R\delta}{2   c_\psi c_{\omega_e} \lambda_{\max}^J}\}$, and we made use of the facts $\|\psi(R_e\T \nabla_{R_e} U(R_e,\theta))\|\leq c_{\psi}$ and $\|\omega_e\| \leq c_{\omega_e}$. Thus, $\mathcal{L}_{\varepsilon}$ is strictly decreasing over the jumps of \eqref{eqn:closed-loop1}. Using similar arguments as the ones used at the end of the proof of Theorem \ref{them:1}, it follows that every maximal solution to \eqref{eqn:closed-loop1} is complete.
	In view of \eqref{eqn:L_varepsilonx_bound}, \eqref{eqn:dL_varepsilonx} and \eqref{eqn:L_varepsilonx+},  one can show that 
	$
	\mathcal{L}_{\varepsilon}(x(t,j)) \leq  \exp(-\lambda t) \mathcal{L}_{\varepsilon}(x(0,0))  \leq  \exp(\lambda J_M)\exp(-\lambda (t+j)) \mathcal{L}_{\varepsilon}(x(0,0))
	$
	for all $(t,j)\in \dom x$ and $x(0,0)\in  \varOmega_c  \times \mathcal{W}_d \subseteq \varOmega_{\mathcal{L}}$ with $\lambda: =  {c_\eta  \lambda_{\min}^{P_3}}/{\lambda_{\max}^{P_2}}$ and $J_M$ denoting the maximum number of jumps. Letting $k := \exp( \lambda J_M) \lambda_{\max}^{P_2}/\lambda_{\min}^{P_1} $ and making use of \eqref{eqn:L_varepsilonx_bound}, one concludes that $|x(t,j)|_{\mathcal{A}}^2 \leq k \exp(-\lambda(t+j))|x(0,0)|_{\mathcal{A}}^2$ 	for all $(t,j)\in \dom x$. This completes the proof.

	\subsection{Proof of Lemma \ref{lem:W}} \label{sec:lemW} 
	
	From the definitions of  $\widehat{\mathcal{A}}_o$ and $\Psi_{W}$, one has $\widehat{\mathcal{A}}_o\in \Psi_{W}$ and $\Psi_W\setminus\{\widehat{\mathcal{A}}_o\}=\{(R_e,\theta,\zeta)\in \mho: (R_e,\theta)\in \Psi_{U}\setminus\{\mathcal{A}_o\}, \zeta = 0 \}$. By Assumption \ref{assum:1}, it follows from \eqref{eqn:u_U_delta} that $\mu_U(R_e,\theta) > \delta$ for all $(R_e,\theta, \zeta)\in \Psi_W\setminus\{\widehat{\mathcal{A}}_o\}$. In view of the definitions of $W$ and $\mu_W$ in \eqref{eqn:defu_W}-\eqref{eqn:defW}, for any $(R_e,\theta, \zeta)\in \Psi_W\setminus\{\widehat{\mathcal{A}}_o\}$ one can show that 
	\begin{align}
		& \mu_W(R_e,\theta,\zeta) \nonumber \\
		&\quad = W(R_e,\theta,0) - \min_{{\theta}'\in \Theta}W(R_e,{\theta}',0) \nonumber \\
		&\quad = U(R_e,\theta) +   {\varrho}\|\psi (R_e\T \nabla_{R_e} U(R_e,{\theta}))\|^2  \nonumber \\
		&\qquad - \min_{{\theta}'\in \Theta} \left( U(R_e,{\theta}')   +  {\varrho}\|\psi (R_e\T \nabla_{R_e} U(R_e,{\theta}'))\|^2 \right) \nonumber\\	 
		&\quad\geq  \mu_U(R_e,\theta) -   {\varrho}c_\psi^2    > \delta -   {\varrho}  c_\psi^2 
	\end{align}
	where we made use of the facts   $\|\psi (R\T \nabla_{R} U(R,\theta))\|=0$ for all $(R,\theta)\in \Psi_{U}\setminus\{\mathcal{A}_o\}$ and $\|\psi (R\T \nabla_{R} U(R,\theta))\| \leq c_\psi$ for all $(R,\theta)\in SO(3)\times \mathbb{R}$ thanks to Assumption \ref{assum:4}. By choosing $\varrho< {(\delta-\delta')}/{c_\psi^2} $, one concludes \eqref{eqn:mu_W_delta}. This completes the proof.
	
	\subsection{Proof of Theorem \ref{them:2}} \label{sec:them2}
	Consider the following Lyapunov function candidate:
	\begin{equation}
		\widehat{\mathcal{L}}(\hat{x}) =  k_R W(R_e,\theta,\zeta)  + \frac{1}{2} \omega_e\T J \omega_e \label{eqn:Lx2}
	\end{equation}
	where $W$  in \eqref{eqn:defW} is a potential function on $\mho$ with respect to  $\widehat{\mathcal{A}}_o$. For the sake of simplicity, we will use the following notations $\psi := \psi(R_e\T \nabla_{R_e} U(R_e,\theta))$ and $\dot{\psi}:=\frac{d}{dt}\psi(R_e\T \nabla_{R_e} U(R_e,\theta))$. From Assumption \ref{assum:1}, it follows that $\psi  =0$ for all $(R_e,\theta)\in \mathcal{A}_o$. Hence, one can verify that $\widehat{\mathcal{L}}$ is positive definite on $\widehat{\mathcal{S}}$ with respect to $\widehat{\mathcal{A}}$. In view of \eqref{eqn:Sigma_def}, \eqref{eqn:dpsi}, \eqref{eqn:h_def}  and \eqref{eqn:hatF}, the time derivative of $\widehat{\mathcal{L}}$ along the flows of \eqref{eqn:closed-loop2} is given by 
	\begin{align}
		\dot{\widehat{\mathcal{L}}}(\hat{x}) & = 2 k_R \omega_e\T \psi   - k_R k_\theta |\nabla_{\theta} U(R_e,\theta)|^2  - \omega_e\T (2 k_R \zeta + k_\omega  \omega_e)  \nonumber\\
		& ~~~ +     {2k_R }{\varrho} (\zeta-\psi)\T ( -\dot{\psi}    -   {k}_\zeta (\zeta- \psi)) \nonumber\\
		&\leq  -k_\omega \|\omega_e\|^2   -  k_R k_\theta |\nabla_{\theta} U(R_e,\theta) |^2   -    {2k_R k_\zeta }{\varrho}  \|\zeta-\psi\|^2 \nonumber \\
		&\quad + 2k_R \|\psi-\zeta\| \|\omega_e\| +  {2k_R }{\varrho} \|\zeta-\psi\| \| \dot{\psi}\|   \nonumber \\
		& \leq  -\eta\T  {P_4} \eta  
	\end{align}
	where $\eta=[\|\omega_e\|,|\nabla_{\theta} U(R_e,\theta)|,\|\psi-\zeta\|]\T$ and  
	\begin{align*}
		P_4: = \begin{bmatrix}
			k_\omega  & 0 & -k_R(1+ {  c_R}{\varrho} )\\
			0 & k_R k_\theta & -{k_R   c_\theta k_\theta}{\varrho} \\
			-k_R(1+ {  c_R}{\varrho} ) & - {k_R   c_\theta k_\theta}{\varrho} &  {2k_R k_\zeta }{\varrho}
		\end{bmatrix}.
	\end{align*} 
	Similar to the matrix $P_3$ in \eqref{eqn:dL_varepsilonx_1}, to guarantee that the matrix $P_4$ is positive definite, it is sufficient to choose
	$
	k_\zeta > \max \left\{{k_R(1 + \varrho c_R)^2}/{\varrho k_\omega},  {c_\theta^2 k_\theta}{\varrho}\right\}:= k_\zeta^*.
	$
	Hence, the time derivative of $\widehat{\mathcal{L}}$   along the flows of \eqref{eqn:closed-loop2} can be rewritten as
	\begin{align}
		\dot{\widehat{\mathcal{L}}}(\hat{x})  
		& \leq  - \lambda_{\min}^{P_4} \|\eta\|^2  \leq 0, \quad  \forall \hat{x}\in \widehat{\mathcal{F}}_c. \label{eqn:dotLx2} 
	\end{align}  
	Thus, $\widehat{\mathcal{L}}$ is non-increasing along the flows of \eqref{eqn:closed-loop2}. Moreover, in view of \eqref{eqn:defW}-\eqref{eqn:closed-loop2} and \eqref{eqn:Lx2}, for any $\hat{x}\in \widehat{\mathcal{J}}_c$, one has $\hat{x}^+ = (R_e, {\theta}^+,\zeta,\omega_e,R_r,\omega_r) $ with $ {\theta}^+ \in  {g}(R_e,\theta)$, and  
	\begin{align}
		\widehat{\mathcal{L}}(\hat{x}^+) - \widehat{\mathcal{L}}(\hat{x}) 
		&= -k_R  ( W(R_e,\theta,\zeta) -     W(R_e,{\theta}^+,\zeta)  )  \nonumber \\ 	 
		&=  - k_R \mu_W(R_e,\theta,\zeta)  \nonumber \\
		&< -k_R \delta'   \label{eqn:Lx+2} 
	\end{align}
	where we made use of \eqref{eqn:mu_W_delta} in Lemma \ref{lem:W}.
	Thus, $\widehat{\mathcal{L}}$ is strictly decreasing over the jumps of the hybrid system \eqref{eqn:closed-loop2}.  It follows from \eqref{eqn:dotLx2} and \eqref{eqn:Lx+2} that the set $ \widehat{\mathcal{A}}$ is stable as per \cite[Theorem 23]{goebel2009hybrid}, and the maximum number of jumps is given by
	$
	J_M:=   \lceil  {\widehat{\mathcal{L}}(\hat{x}(0,0))}/{(k_R \delta')} \rceil
	$. 
	Moreover, applying the invariance principle in \cite[Theorem 4.7]{sanfelice2007invariance}, any maximal solution to \eqref{eqn:closed-loop2} must converge to the largest invariant set contained in
	$
	\widehat{\mathcal{W}}:
	=\{\hat{x}\in \widehat{\mathcal{F}}_c  : \nabla_{\theta} U(R_e,\theta)=0,  \omega_e=0, 
	\zeta=\psi \}.
	$
	From $\omega_e\equiv 0$, one has $\dot{\omega}_e \equiv 0$, which in view of \eqref{eqn:Jomega_e} and \eqref{eqn:hybrid_feedback2},  implies that $\zeta = \psi(R_e\T \nabla_{R_e} U(R_e,\theta))  = 0$. Then, it follows from $ \|\zeta\|= \|\psi(R_e\T \nabla_{R_e} U(R_e,\theta))\| = \nabla_{\theta} U(R_e,\theta)=0$ that $(R_e,\theta,\zeta)\in \Psi_W$. Thus, any solution  to \eqref{eqn:closed-loop2} must converge to  the largest invariant set contained in $\widehat{ {\mathcal{W}}}'
	:= \{\hat{x}\in \widehat{\mathcal{F}}_c ~|~  (R_e,\theta,\zeta)\in \widehat{\mathcal{F}} \cap \Psi_W,   \omega_e = 0 \}$. Similar to the proof of Theorem \ref{them:1}, applying simple set-theoretic arguments, one obtains $\widehat{\mathcal{F}} \cap \Psi_W = \{\widehat{\mathcal{A}}_o\}$ and $\widehat{\mathcal{W}}' = \widehat{\mathcal{A}}$. Moreover, following similar arguments as the ones used at the end of the proof of Theorem \ref{them:1}, one can show that  every maximal solution to \eqref{eqn:closed-loop2} is complete. Finally, one can conclude that the set $\widehat{\mathcal{A}}$ is globally asymptotically stable for the hybrid system \eqref{eqn:closed-loop2}.  This completes the proof.

	\subsection{Proof of Theorem \ref{them:3}} \label{sec:them3}
	Consider the following Lyapunov function candidate:
	\begin{equation}
		\overline{\mathcal{L}}(\bar{x}) =  k_R U(R_e,\theta) + k_\beta U(\tilde{R},\bar{\theta}) + \frac{1}{2} \omega_e\T J \omega_e. \label{eqn:barLx}
	\end{equation}
	Since $U$ is a potential function on $SO(3)\times \mathbb{R}$, one can  verify that $\overline{\mathcal{L}}$ is positive definite on $\overline{\mathcal{S}}$ with respect to $\overline{\mathcal{A}}$.  The time derivative of $\overline{\mathcal{L}}$ along the flows of \eqref{eqn:closed-loop3} is given by
	\begin{align}
		\dot{\overline{\mathcal{L}}}(\bar{x})
		&=   k_R\dot{U}(R_e,\theta) +  k_\beta \dot{U}(\tilde{R},\bar{\theta})   - \omega_e\T  \bar{\kappa}(R_e,\theta,\tilde{R},\bar{\theta}) \label{eqn:dotbarL}
	\end{align}
	where we made use of the fact $\omega_e\T  \Sigma(R_e,\omega_e,\omega_r)\omega_e =0$. From \eqref{eqn:Re} and \eqref{eqn:bar_R}, one obtains $\dot{\tilde{R}} = \tilde{R}(\omega_e-\beta)^\times$.  From \eqref{eqn:dynamics_theta}, \eqref{eqn:ftheta} and \eqref{eqn:bar_R}, one obtains
	\begin{multline}
		\dot{U}(\tilde{R},\bar{\theta})
		= 2 (\omega_e-\beta)\T \psi (\tilde{R}\T \nabla_{\tilde{R}} U(\tilde{R},\bar{\theta}))  \\
		-   k_\theta |\nabla_{\bar{\theta}} U(\tilde{R},\bar{\theta})|^2.   \label{eqn:dotU3}
	\end{multline}
	Substituting \eqref{eqn:beta}, \eqref{eqn:kappa2}, \eqref{eqn:dotU_Re} and \eqref{eqn:dotU3} into \eqref{eqn:dotbarL}, the time derivative of $\overline{\mathcal{L}}$ along the flows of \eqref{eqn:closed-loop3} can be rewritten as
	\begin{align}
		\dot{\overline{\mathcal{L}}}(\bar{x})
		& =  - k_R k_\theta |\nabla_{\theta} U(R_e,\theta)|^2 -  k_\beta k_\theta |\nabla_{\bar{\theta}} U(\tilde{R},\bar{\theta})|^2\nonumber\\ 
		&~~~ - 2k_\beta  \psi (\tilde{R}\T \nabla_{\tilde{R}} U(\tilde{R},\bar{\theta}))\T \Gamma  \psi (\tilde{R}\T \nabla_{\tilde{R}} U(\tilde{R},\bar{\theta}))   \label{eqn:dotbarLx} 
	\end{align}
	for all $\bar{x}\in \overline{\mathcal{F}}_c$. 
	Since the matrix $\Gamma$ is symmetric positive definite, $\dot{\overline{\mathcal{L}}}$ is negative semi-definite in the flow set and $\overline{\mathcal{L}}$ is non-increasing along the flows of \eqref{eqn:closed-loop3}. For any $\bar{x}  \in\overline{\mathcal{J}}_c$, one obtains $\bar{x}^+  \in (R_e, g(R_e,\theta),\omega_e,R_r,\omega_r,\tilde{R},\bar{\theta}) $ if $\bar{x}   \in\overline{\mathcal{J}}_{c1}\setminus\overline{\mathcal{J}}_{c2}$,  
	$\bar{x}^+ \in (R_e,\theta,\omega_e,R_r,\omega_r,\tilde{R},g(\tilde{R},\bar{\theta})) $ if   $\bar{x}   \in\overline{\mathcal{J}}_{c2}\setminus\overline{\mathcal{J}}_{c1}$, or $\bar{x}^+ \in (R_e,g(R_e,\theta),\omega_e,R_r,\omega_r,\tilde{R},g(\tilde{R},\bar{\theta})) $ if $\bar{x}   \in\overline{\mathcal{J}}_{c1}\cap \overline{\mathcal{J}}_{c2}$.
	Similar to \eqref{eqn:Lx+}, in view of \eqref{eqn:closed-loop3} and \eqref{eqn:barLx}, one can show that
	\begin{align}
		\overline{\mathcal{L}}(\bar{x}^+) - \overline{\mathcal{L}}(\bar{x})  
		&  < - k^* \delta \label{eqn:barLx+}
	\end{align}
	for all $\bar{x}  \in\overline{\mathcal{J}}_c$  with $k^* := \min\{k_R,k_\beta\}$. 
	Thus, $\overline{\mathcal{L}}$ is strictly decreasing over the jumps of \eqref{eqn:closed-loop3} on $\overline{\mathcal{J}}_c$.  Similar to the proof of Theorem \ref{them:1}, from \eqref{eqn:dotbarLx} and \eqref{eqn:barLx+}   one concludes that the set $\overline{\mathcal{A}}$ is stable as per \cite[Theorem 23]{goebel2009hybrid},  and the number of jumps is bounded by $J_M:=\left\lceil {\overline{\mathcal{L}}(\bar{x}(0,0))}/{(k^* \delta)}\right\rceil$. 
	
	Next, we will show the global attractivity of set $\overline{\mathcal{A}}$. Applying the invariance principle for hybrid systems given in \cite[Theorem 4.7]{sanfelice2007invariance}, one obtains that every solution $\bar{x}$ to \eqref{eqn:closed-loop3} must converge to the largest invariant set contained in
	$
	\overline{\mathcal{W}}
	:=\{\bar{x}\in \overline{\mathcal{F}}_c : \nabla_{\theta} U(R_e,\theta)=0,  \psi(\tilde{R}\T \nabla_{\tilde{R}} U(\tilde{R},\bar{\theta}))=0,  
	\nabla_{\bar{\theta}} U(\tilde{R},\bar{\theta})=0\}.
	$
	For each $\bar{x}\in \overline{\mathcal{W}}$, it follows that $\nabla_{\theta} U(R_e,\theta)=0$ and  $(\tilde{R},\bar{\theta})\in \mathcal{F} \cap \Psi_U$. Similar to the proof of Theorem \ref{them:1}, one has $(\tilde{R},\bar{\theta})\in \mathcal{F} \cap \Psi_U = \{\mathcal{A}_o\}$. From $\tilde{R} \equiv I_3$ one obtains  $\dot{\tilde{R}} = 0$ and $ \omega_e-\beta = 0$. Recall the definition of $\beta$ in \eqref{eqn:beta}, it follows from $(\tilde{R},\bar{\theta})= \mathcal{A}_o$  that  $\omega_e = \beta = 0$. From $\omega_e\equiv 0$, one has $\dot{\omega}_e = 0$. Since $\omega_e= \dot{\omega}_e =0$ and $\psi (\tilde{R}\T \nabla_{\tilde{R}} U(\tilde{R},\bar{\theta}))= 0$, it follows from  \eqref{eqn:Jomega_e}, \eqref{eqn:hybrid_feedback3} and \eqref{eqn:kappa2}  that $\psi(R_e\T \nabla_{R_e} U(R_e,\theta)) = 0$. Using this fact, together with $ \nabla_{\theta} U(R_e,\theta)= 0$, one can show that $(R_e,\theta)= \mathcal{A}_o$. Hence, one verifies that $\overline{\mathcal{W}} = \overline{\mathcal{A}}$ from the definitions of $\overline{\mathcal{W}}$ and $\overline{\mathcal{A}}$. Using similar arguments as the ones used at the end of the proof of Theorem \ref{them:1}, it follows that every maximal solution to \eqref{eqn:closed-loop3} is complete.
	Finally, one can conclude that the set $\overline{\mathcal{A}}$ is globally asymptotically stable for the hybrid system \eqref{eqn:closed-loop3}.  This completes the proof.

	\subsection{Proof of Lemma \ref{lemma:U}} \label{sec:lemma_U}
	From \eqref{eqn:transfT}, the time derivative of the transformation map $\mathcal{T}$ along the trajectories of $\dot{R}=R\omega^\times$ and $\dot{\theta}=\nu$ is given by
	\begin{align*}
		\dot{\mathcal{T}}(R,\theta)  &= R \omega^\times  \mathcal{R}_a(\theta,u) + \nu R 		\mathcal{R}_a(\theta,u)u^\times  \\
		&= \mathcal{T}(R,\theta) (\mathcal{R}_a(\theta,u)\T \omega + \nu u)^\times
	\end{align*}
	where we made use of the facts: $\mathcal{R}_a(\theta,u)=\exp(\theta u^\times)$ and $ \dot{\mathcal{R}}_a(\theta,u) = \frac{d}{dt}\exp(\theta u^\times) =  \mathcal{R}_a(\theta,u) \nu u^\times$. The gradients $\nabla_{R} U(R,\theta)$ and $\nabla_{\theta} U(R,\theta)$ can be computed from the differential of $U$ in an arbitrary tangential direction $(R\omega^\times,\nu) \in T_R SO(3) \times \mathbb{R}$, which is given as
	\begin{align}
		\dot{U}(R,\theta)
		& =  \langle \nabla_{R} U(R,\theta), R \omega^\times   \rangle_R  +  \langle \langle  \nabla_{\theta} U(R,\theta), \nu \rangle \rangle\nonumber \\
		& =  \langle\langle R\T \nabla_{R} U(R,\theta), \omega^\times \rangle \rangle  +  \langle \langle \nabla_{\theta} U(R,\theta), \nu \rangle \rangle \nonumber\\
		& = 2 \omega\T \psi (R\T \nabla_{R} U(R,\theta)) +  \nu \nabla_{\theta} U(R,\theta)  \label{eqn:dotU}
	\end{align}
	where we made use of the property $\langle\langle A, x^\times\rangle\rangle = 2x\T \psi (A) $. On the other hand, from \eqref{eqn:URtheta} and \eqref{eqn:dTRtheta} the time derivative of $U$ can be directly obtained as
	\begin{align}
		\dot{U}(R,\theta) 
		& = \tr(-A\mathcal{T}(R,\theta) (\mathcal{R}_a(\theta,u)\T \omega + \nu u)^\times) + \gamma \theta \nu \nonumber \\
		& = \langle\langle A\mathcal{T}(R,\theta), (\mathcal{R}_a(\theta,u)\T  \omega + \nu u)^\times  \rangle\rangle  + \gamma \theta \nu \nonumber \\
		& = \langle\langle \mathbb{P}_a(A\mathcal{T}(R,\theta)), (\mathcal{R}_a(\theta,u)\T  \omega + \nu u)^\times \rangle\rangle   + \gamma \theta \nu \nonumber \\
		& = 2\omega\T \mathcal{R}_a(\theta,u) \psi(A\mathcal{T}(R,\theta))   \nonumber \\
		&\qquad  \qquad \qquad + \nu \left( 2u\T \psi (A\mathcal{T}(R,\theta))  + \gamma \theta  \right)
		\label{eqn:dotU2}
	\end{align}
	where we made use of the facts: $(x^\times)\T = - x^\times$, $\tr(A\T B) = \langle\langle A, B\rangle\rangle $ and $ \langle\langle A, x^\times \rangle\rangle = 2x\T \psi (A)$  and $\psi(\mathbb{P}_a(A)) = \psi(A)$ for all $x\in \mathbb{R}^3, A,B\in \mathbb{R}^{3\times 3}$. In view of \eqref{eqn:dotU} and \eqref{eqn:dotU2}, one can easily obtain \eqref{eqn:gradient_R_U} and \eqref{eqn:gradient_theta_U}. 
	
	In view of \eqref{eqn:gradient_R_U} and \eqref{eqn:gradient_theta_U}, it follows from $ |\nabla_{\theta} U(R,\theta)| =0$  and $ \|\psi (R\T \nabla_{R} U(R,\theta))\|=0$   that $ \|\psi(A\mathcal{T}(R,\theta))\| =\theta = 0$. Recall the definition of $\mathcal{T}(R,\theta)$   in \eqref{eqn:transfT}, one can further show that $\psi(A\mathcal{T}(R,\theta))=\psi(AR)=0$ since $\mathcal{T}(R,\theta)=I_3$ as $\theta=0$. Using the fact $\psi(AR)=0$, one obtains $\mathbb{P}_a(AR)=0$, which implies that $AR=R\T A$ from the definition of the map $\mathbb{P}_a$.  Applying \cite[Lemma 2]{mayhew2011synergistic}, one obtains $R\in \Psi_V  = \{I_3\} \cup \mathcal{R}_a(\pi,\mathcal{E}(A))$ with $\mathcal{E}(A)$ denoting the set of eigenvectors of $A$. Using this result, together with $\theta=0$, one can conclude that the set of all the critical points of  $U(R,\theta)$   in \eqref{eqn:defu_U} is given as $\Psi_U = \Psi_V \times \{0\}$ and   $\mathcal{A}_o\in \Psi_U$, which gives \eqref{eqn:crit_U_R}.
	
	On the other hand, applying the properties of $\psi$ given in \cite[Lemma 1]{berkane2017hybrid2}, the time derivative of $\psi(A\mathcal{T}(R,\theta))$ is given by $\dot{\psi}(A\mathcal{T}(R,\theta)) = E(A\mathcal{T}(R,\theta))  ( \mathcal{R}_a\T(\theta,u)\omega +  vu )$ along the trajectories of $\dot{R}=R\omega^\times$ and $\dot{\theta}=v$. Then, in view of  \eqref{eqn:dTRtheta}-\eqref{eqn:gradient_R_U}, the time derivative of  $ \psi(R\T \nabla_{R} U(R,\theta))$   is given by
	\begin{align}
		&\dot{\psi}(R\T \nabla_{R} U(R,\theta)) \nonumber \\
		& \quad=  \dot{\mathcal{R}}_a(\theta,u) \psi(A\mathcal{T}(R,\theta)) +  \mathcal{R}_a(\theta,u) \dot{\psi}(A\mathcal{T}(R,\theta)) \nonumber \\
		& \quad  = -\left( \mathcal{R}_a(\theta,u) \psi(A\mathcal{T}(R,\theta))\right) ^\times  vu \nonumber \\
		&  \qquad ~ + \mathcal{R}_a(\theta,u)E(A\mathcal{T}(R,\theta))  ( \mathcal{R}_a\T(\theta,u)\omega +  vu ) \nonumber \\
		&  \quad =  \mathcal{D}_{R} (R,\theta) \omega + \mathcal{D}_{\theta} (R,\theta) v.  
	\end{align}
	This completes the proof.

	\subsection{Proof of Proposition \ref{prop:3}} \label{sec:prop3}

	For the sake of simplicity, let $\mathcal{T}=\mathcal{T}(R,\theta)$.
	From   \cite[Lemma 2]{berkane2017hybrid2}, one has the following properties for any $\mathcal{T}\in SO(3)$:
	\begin{align}
		4 \lambda_{\min}^{\bar{A}} |\mathcal{T}|_I^2 &\leq \tr(A(I-\mathcal{T})) \leq 4 \lambda_{\max}^{\bar{A}} |\mathcal{T}|_I^2  \label{eqn:property_tr} \\
		\|\psi (A \mathcal{T})\|^2 & = \alpha_A(\mathcal{T}) \tr(\underline{A}(I_3-\mathcal{T})) \label{eqn:property_psi} 
	\end{align}	
	where matrices $\bar{A}  = \frac{1}{2}(\tr(A)I_3-A), \underline{A} = \tr(\bar{A}^2)I_3 - 2\bar{A}^2$ are symmetric positive definite as matrix $A$ is symmetric positive definite, and $\alpha_A(\mathcal{T})  = 1-|\mathcal{T}|_I^2 \cos^2\measuredangle(u,\bar{A}u)$ with $\measuredangle(~,~)$ denoting the angle between two vectors and $u$ denoting the axis of the rotation matrix $\mathcal{T}$. Using the facts $\frac{1}{2}(\tr(\underline{A})I_3-\underline{A}) = \bar{A}^2$ and $\alpha_A(\mathcal{T})<1, \forall \mathcal{T}\in SO(3)$,  one obtains from \eqref{eqn:property_tr} that
	\begin{align} 
		4\alpha_A(\mathcal{T}) (\lambda_{\min}^{\bar{A}})^2 |\mathcal{T}|_I^2&\leq  \|\psi (A \mathcal{T})\|^2 \leq 4  (\lambda_{\max}^{\bar{A}})^2 |\mathcal{T}|_I^2 \label{eqn:property_psi2}. 
	\end{align} 
	From \eqref{eqn:gradient_R_U} and  \eqref{eqn:property_psi2},  one can   show that Assumption \ref{assum:3} holds by choosing $c_\psi\geq 2 \lambda_{\max}^{\bar{A}}$ since that $\|\psi(R\T \nabla_{R} U(R,\theta))\|^2 \leq \|\mathcal{R}_a(\theta,u) \psi(A\mathcal{T})\| \leq 4  (\lambda_{\max}^{\bar{A}})^2 |\mathcal{T}|_I^2 \leq 4  (\lambda_{\max}^{\bar{A}})^2$ for all $(R,\theta)\in SO(3)\times \mathbb{R}$. 
	
	Next, we are going to verify the conditions in Assumption \ref{assum:2}. 
	From \eqref{eqn:URtheta}, \eqref{eqn:gradient_R_U}-\eqref{eqn:gradient_theta_U} and \eqref{eqn:property_tr}-\eqref{eqn:property_psi2}, one can show that 
	\begin{align} 
		& \|\psi(R\T \nabla_{R} U(R,\theta))\|^2 + |\nabla_{\theta} U(R,\theta)|^2 \nonumber \\
		& \quad = \| \psi(A\mathcal{T})\|^2 + | \gamma \theta  |^2 + 4|u \T  \psi (A \mathcal{T})|^2 + 4\gamma \theta  u \T  \psi (A \mathcal{T}) \nonumber \\
		& \quad \leq 7\|\psi(A\mathcal{T})\|^2 + 3 \gamma^2 |\theta|^2   \nonumber \\
		& \quad\leq 28  (\lambda_{\max}^{\bar{A}})^2 |\mathcal{T}|_I^2 + 3\gamma^2 |\theta|^2 \nonumber \\
		& \quad\leq \frac{7 (\lambda_{\max}^{\bar{A}})^2}{\lambda_{\min}^{\bar{A}}} \tr(A(I-\mathcal{T})) +  6 \gamma \left(  \frac{\gamma}{2}   |\theta|^2 \right)  \nonumber \\
		&\quad \leq \alpha_1 U(R,\theta) , \quad \forall (R,\theta) \in SO(3)\times \mathbb{R}
	\end{align}
	where $\alpha_1: = \max\{\frac{7  (\lambda_{\max}^{\bar{A}})^2}{\lambda_{\min}^{\bar{A}}},6\gamma\}$, and we made use of the facts $\|u\|=1$, $ |u \T  \psi (A \mathcal{T})|  \leq   \|u \|  \|\psi (A \mathcal{T})\|  = \|\psi (A \mathcal{T})\|$ and $4\gamma \theta  u \T  \psi (A \mathcal{T}) \leq 4|\gamma \theta| \| \psi (A \mathcal{T})\|  \leq  2|\gamma \theta|^2 + 2\|\psi (A \mathcal{T})\|^2$. On the other hand, from  the definition  of $\mathcal{F}$ in \eqref{eqn:Fset}, one has $(R,\theta)\notin \Psi_U\setminus\{\mathcal{A}_o\}$ for all $(R,\theta) \in \mathcal{F}$. This implies that $\mathcal{T}\notin \Psi_V\setminus\{I_3\}=  \mathcal{R}_a(\pi,\mathcal{E}(\bar{A}))$. Hence, one obtains $\alpha_A(\mathcal{T})  = 1-|\mathcal{T}|_I^2 \cos^2\measuredangle(u,\bar{A}u)>0$ for all $(R,\theta) \in \mathcal{F}$. Letting $\alpha_A^*: = \inf_{(R,\theta) \in \mathcal{F}}\alpha_A(\mathcal{T})>0$, it follows from \eqref{eqn:property_psi2} that $\|\psi (A \mathcal{T})\|^2\geq 4\alpha_A^* (\lambda_{\min}^{\bar{A}})^2 |\mathcal{T}|_I^2 $ for all $(R,\theta) \in \mathcal{F}$. From \eqref{eqn:URtheta}, \eqref{eqn:gradient_R_U}-\eqref{eqn:gradient_theta_U} and \eqref{eqn:property_tr}-\eqref{eqn:property_psi2}, one can show that 
	\begin{align} 
		& \|\psi(R\T \nabla_{R} U(R,\theta))\|^2 + |\nabla_{\theta} U(R,\theta)|^2 \nonumber \\
		& \quad \geq \|\mathcal{R}_a(\theta,u) \psi(A\mathcal{T})\|^2 + \frac{1}{8}| \gamma \theta + 2u \T  \psi (A \mathcal{T})|^2 \nonumber \\
		& \quad \geq  \| \psi(A\mathcal{T})\|^2 + \frac{1}{8}\left( \frac{1}{2} | \gamma \theta  |^2  - 4 |u \T  \psi (A \mathcal{T}) |^2 \right)  \nonumber \\
		& \quad \geq \frac{1}{2}\|\psi(A\mathcal{T})\|^2 + \frac{1}{16} \gamma^2 |\theta|^2   \nonumber \\
		& \quad \geq 2\alpha_A^* (\lambda_{\min}^{\bar{A}})^2 |\mathcal{T}|_I^2 + \frac{1}{16} \gamma^2 |\theta|^2   \nonumber \\
		& \quad \geq \frac{\alpha_A^* (\lambda_{\min}^{\bar{A}})^2}{2 \lambda_{\max}^{\bar{A}}}  \tr(A(I-\mathcal{T})) + \frac{\gamma}{8}     \left(  \frac{\gamma}{2}   |\theta|^2 \right)  \nonumber \\
		&\quad \geq \alpha_2 U(R,\theta), \quad \forall (R,\theta) \in \mathcal{F} 
	\end{align}
	where  $\alpha_2:=\min\{\frac{\alpha_A^* (\lambda_{\min}^{\bar{A}})^2}{2 \lambda_{\max}^{\bar{A}}}, \frac{\gamma}{8}\}$, and we made use of the facts: $|u\T \psi (A \mathcal{T})| \leq \|u\| \|\psi (A \mathcal{T})\| \leq \|\psi (A \mathcal{T})\|$, $4 \gamma \theta  u \T  \psi (A \mathcal{T}) \geq -4 |\gamma \theta| |u\T \psi (A \mathcal{T})| \geq - \frac{1}{2} |\gamma \theta|^2 - 8|u\T \psi (A \mathcal{T})|^2$. From the definitions of $\alpha_1,\alpha_2$ and using the fact $6\gamma>\frac{\gamma}{8},\forall \gamma>0$, it is clear that $\alpha_1>\alpha_2$.

	Now, we are going to verify the conditions in Assumption \ref{assum:4}.  Applying the definitions of $\mathcal{D}_{R}(R,\theta)$ and $\mathcal{D}_{\theta}(R,\theta)$,  for each $(R,\theta)\in SO(3) \times \mathbb{R}$ one can show that   $\|\mathcal{D}_{R} (R,\theta)\|_F   = \|E(A\mathcal{T}(R,\theta)) \|_F \leq  \|\bar{A}\|_F$  and $
	\|\mathcal{D}_{\theta} (R,\theta)\|   \leq  \|E(A\mathcal{T}(R,\theta)) \|_F + \|\psi(A\mathcal{T}(R,\theta) \| \leq  \|\bar{A}\|_F + 2 \lambda_{\max}^{\bar{A}}  $ using the facts: $\|u\|=1$,  $\|E(AR) \|_F\leq \|\bar{A}\|_F$ and $\|\psi(AR\| \leq  2 \lambda_{\max}^{\bar{A}}$ for any $R\in SO(3)$ as per \cite[Lemma 2]{berkane2017hybrid2}. It follows from \eqref{eqn:ftheta} and \eqref{eqn:dot_psi} that 
	\begin{align*}
		&\|\dot{\psi}(R_e\T \nabla_{R_e} U(R_e,\theta))\|  \\ 
		&\quad  \leq  \|\mathcal{D}_{R} (R_e,\theta) \omega_e\| + k_\theta \|\mathcal{D}_{\theta} (R_e,\theta)   \| |\nabla_{\theta}U(R_e,\theta)| \\
		&\quad \leq  \|\bar{A}\|_F \|\omega\| + k_\theta (\|\bar{A}\|_F + 2 \lambda_{\max}^{\bar{A}}) |\nabla_{\theta}U(R_e,\theta)|
	\end{align*}
	for all $(R_e,\theta)\in SO(3) \times \mathbb{R}$.
	By choosing $c_R \geq \|\bar{A}\|_F$ and $c_\theta\geq \|\bar{A}\|_F + 2 \lambda_{\max}^{\bar{A}}$, one can conclude that inequality \eqref{eqn:dpsi}  is satisfied for all $(R_e,\theta)\in \mathcal{F}$.

	Finally, we are going to verify the conditions in Assumption \ref{assum:1}. From \eqref{eqn:crit_U_R} and $\Psi_V  = \{I_3\} \cup \mathcal{R}_a(\pi,\mathcal{E}(A))$, one obtains that $\Psi_U\setminus\{\mathcal{A}_o\} = \{(R,\theta)\in SO(3)\times \mathbb{R}: R = \mathcal{R}_a(\pi,v), v\in \mathcal{E}(A) , \theta = 0 \}$. Let $\lambda^{\bar{A}}_v$ be the eigenvalue of $\bar{A}$ associated to the eigenvector $v\in \mathcal{E}(\bar{A}) \equiv \mathcal{E}(A)$.
	For any   $v\in \mathcal{E}(A)$ and $\theta \in \mathbb{R}$, one can show that
	\begin{align}
		U(\mathcal{R}_a(\pi,v),0)  & = V(\mathcal{R}_a(\pi,v))  = 4 v\T \bar{A} v = 4\lambda_v^{\bar{A}}  \label{eqn:Upiv} \\
		U(\mathcal{R}_a(\pi,v),\theta) & = V(\mathcal{R}_a(\pi,v)\mathcal{R}_a(\theta,u)) + \frac{\gamma}{2} \theta^2 \nonumber \\
		&= V(\mathcal{R}_a(\pi,v))  + \frac{\gamma}{2} \theta^2  \nonumber \\
		&~~~ + \tr(A\mathcal{R}_a(\pi,v)(I-\mathcal{R}_a(\theta,u))) \nonumber \\
		&   = 4\lambda_v^{\bar{A}} + \frac{\gamma}{2} \theta^2 - 2\sin^2\left(\frac{\theta}{2}\right)  \Delta(v,u)  \label{eqn:Upiv2}
	\end{align}
	where $\Delta(u,v) = u\T \left(\tr(A)I - A - 2v\T A v(I_3 - vv\T) \right) u$, and we made use of the facts: $\mathcal{R}_a(\theta,u)  =  I_3 + \sin(\theta) u^\times + (1-\cos(\theta))(u^\times)^2$, $A\mathcal{R}_a(\pi,v) = A(I+2(v^\times)^2) = 2Avv\T -A$ and $ \tr(A\mathcal{R}_a(\pi,v)(I-\mathcal{R}_a(\theta,u))) =  -2\sin^2(\frac{\theta}{2})\Delta(u, v)$. Let $\Delta^*  = \min_{v\in \mathcal{E}(A)}\Delta(v,u) >0$, $\gamma<\frac{ 4\Delta^* }{\pi^2}$ and $\delta <  (  \frac{4 \Delta^* }{\pi^2} -\gamma )  \frac{\theta_M^2}{2}$. In view of \eqref{eqn:defu_U},  \eqref{eqn:Upiv} and \eqref{eqn:Upiv2}, for any $(R,\theta)\in \Psi_U\setminus\{\mathcal{A}_o\} $, one can show that
		\begin{align*}
			\mu_U(R,\theta) &= U(\mathcal{R}_a(\pi,v),0) - \min_{{\theta}'\in \Theta}U(\mathcal{R}_a(\pi,v),{\theta}') \\
			&  = \max_{{\theta'}\in \Theta}  \left(2\sin^2\left(\frac{{\theta}'}{2}\right)  \Delta(v,u) - \frac{\gamma}{2} {\theta'}^2 \right)  \\
			&\geq 2\sin^2\left(\frac{{\theta}_M}{2}\right)  \Delta(v,u) - \frac{\gamma}{2} \theta_M^2  \\
			&\geq  \left( \frac{ 4\Delta^* }{\pi^2}  -  \gamma  \right)  \frac{\theta_M^2}{2}   > \delta
		\end{align*} 
		where we made use of the facts $\theta_M  = \sup_{\theta'\in \Theta} |\theta'|$, $|\sin(\frac{\theta}{2})| \geq \frac{|\theta|}{\pi}$ and $2\sin^2\left(\frac{{\theta}}{2}\right)  \Delta(v,u) - \frac{\gamma}{2} {\theta}^2\geq 0$ for all $ |\theta| \in [0,\pi]$. 
		Given the set $\mathcal{P}_U$ in \eqref{eqn:PU}, it follows from \cite[Proposition 2]{berkane2017construction} that $\Delta^*>0$. This completes the proof.

		\subsection{Useful properties on $SE(3)$}\label{sec:property_SE(3)}
		In this subsection, we first introduce some definitions of the maps $(\cdot)^\wedge$, $\bar{\psi}$, adjoint action map $\Ad$ and adjoint operator  $\ad$. 
		For all $\xi = [\omega\T,v\T]\T $ with $\omega,v \in \mathbb{R}^3$,  we define the   map $(\cdot)^\wedge: \mathbb{R}^6 \to \mathfrak{se}(3)$ as
		\begin{equation}
			\xi^\wedge = \begin{bmatrix}
				\omega^\times & v\\
				0 & 0
			\end{bmatrix} \in \mathfrak{se}(3).
		\end{equation}
		Motivated by \cite{wang2019hybrid}, we introduce the following map $\bar{\psi}:\mathbb{R}^{4\times 4}\to \mathbb{R}^6$ given as:
		\begin{align}
			\bar{\psi}(\mathbb{A})
			= \begin{bmatrix}
				\psi(A)\\
				\frac{1}{2}b
			\end{bmatrix}, ~~ \forall \mathbb{A} = \begin{bmatrix}
				A & b\\
				c\T & d
			\end{bmatrix}\in \mathbb{R}^{4\times 4}
		\end{align}
		with $A \in \mathbb{R}^{3\times 3}, b,c\in \mathbb{R}^3, d\in \mathbb{R}$. Similar to the map $\psi$,  one has the following identities:
		\begin{align}
			\langle\langle \mathbb{A}, y^\wedge \rangle\rangle &= 2y\T \bar{\psi}(\mathbb{A}) \label{eqn:bar_psi1} \\
			\bar{\psi}(X\T (I_4-X) \mathbb{A}) &= -\bar{\psi} ((I_4 - X^{-1})\mathbb{A}) \label{eqn:bar_psi2}
		\end{align}
		for all $\mathbb{A}\in \mathbb{R}^{4\times 4},y\in \mathbb{R}^6$.	We define the adjoint operator  $\ad: \mathbb{R}^6 \to \mathbb{R}^{6\times 6}$ as
		\begin{equation}
			\ad_\xi = \begin{bmatrix}
				\omega^\times & 0\\
				v^\times & \omega^\times 
			\end{bmatrix} \in  \mathbb{R}^{6\times 6}, \quad\forall \xi = \begin{bmatrix}
				\omega\\
				v
			\end{bmatrix} 
		\end{equation}
		and the adjoint map $\Ad: SE(3) \to \mathbb{R}^{6\times 6}$   as
		\begin{equation}
			\Ad_X = \begin{bmatrix}
				R & 0\\
				p^\times R & R
			\end{bmatrix} \in  \mathbb{R}^{6\times 6}, \quad \forall  X = \begin{bmatrix}
				R & p\\
				0 & 1
			\end{bmatrix}.
		\end{equation}
		One can also verify the following identities:
		\begin{subequations}
			\begin{align}
				\Ad_{X^{-1}} &= \Ad_X^{-1}\\
				\Ad_X \Ad_Y &= \Ad_{XY}  \\
				X x^\wedge X^{-1} &= (\Ad_X x)^\wedge\\
				\det(\Ad_X) & = 1 \\
				\ad_x x &= 0\\
				\ad_x y &= -\ad_y x
			\end{align}
		\end{subequations}	
		for all $X,Y\in SE(3), x,y\in \mathbb{R}^6$.
		Moreover, along the trajectories $\dot{X} = X\xi^\wedge$ with $(X,\xi) \in SE(3)\times \mathbb{R}^6$, one has
			\begin{align}
				\frac{d}{dt} \Ad_X  = \Ad_X \ad_\xi, \quad
				\frac{d}{dt} \Ad_X^{-1}  = -\ad_\xi \Ad_X^{-1}.
			\end{align}

	\bibliographystyle{IEEEtran}
	
	\bibliography{mybib}
\end{document}